\documentclass[preprint]{siamart250211}

\usepackage{lipsum}
\usepackage{amsfonts, amssymb}
\usepackage{graphicx}
\usepackage{epstopdf}
\usepackage{algorithmic}
\usepackage{mathtools}
\usepackage{cleveref}
\usepackage{todonotes}
\usepackage{subcaption}
\usepackage{comment}
\usepackage{bbm}
\usepackage{tikz}
\usepackage{xcolor}
\usepackage{gensymb}
\usetikzlibrary{fadings}
\usetikzlibrary{patterns}
\usetikzlibrary{shadows.blur}
\usetikzlibrary{shapes}
\usetikzlibrary{calc}
\usetikzlibrary{positioning}
\usetikzlibrary{hobby}

\ifpdf
  \DeclareGraphicsExtensions{.eps,.pdf,.png,.jpg}
\else
  \DeclareGraphicsExtensions{.eps}
\fi

\usepackage{enumitem}
\setlist[enumerate]{leftmargin=.5in}
\setlist[itemize]{leftmargin=.5in}

\newsiamremark{remark}{Remark}
\newsiamremark{hypothesis}{Hypothesis}
\crefname{hypothesis}{Hypothesis}{Hypotheses}
\newsiamthm{claim}{Claim}

\headers{Analysis-Aware Defeaturing of Dirichlet Features}{P. Weder and A. Buffa}

\title{Analysis-Aware Defeaturing of Dirichlet Features\thanks{Submitted to the editors 19/08/2025.
\funding{This research was supported by the Swiss National Science Foundation via project MINT n. 200021\_215099, PDE tools for analysis-aware geometry processing in simulation science.}}
}

\author{
Philipp Weder\thanks{Institute of Mathematics, École Polytechnique Fédérale de Lausanne, Switzerland (\email{philipp.weder@epfl.ch}, \email{annalisa.buffa@epfl.ch})}
\and
Annalisa Buffa\footnotemark[2]
}

\usepackage{amsopn}
\usepackage{macros}

\makeatletter
\newcommand*{\addFileDependency}[1]{
  \typeout{(#1)}
  \@addtofilelist{#1}
  \IfFileExists{#1}{}{\typeout{No file #1.}}
}
\makeatother

\ifpdf
\hypersetup{
  pdftitle={Analysis-Aware Defeaturing of Dirichlet Features},
  pdfauthor={P. Weder and A. Buffa}
}
\fi

\begin{document}

\maketitle

\begin{abstract}
Feature removal from computational geometries, or defeaturing, is an integral part of industrial simulation pipelines. Defeaturing simplifies the otherwise costly or even impossible meshing process, speeds up the simulation, and lowers its memory footprint. Current defeaturing operators are often based on heuristic criteria and ignore the impact of the simplifications on the PDE solution. This work extends the mathematically rigorous framework developed in \cite{buffa_analysis-aware_2022} to features subject to Dirichlet boundary conditions in Poisson problems. We derive a posteriori error estimators for negative features in the interior or on the boundary of the computational domain. The estimators' dependence on the feature size is explicit, and their evaluation only involves boundary integrals over the feature boundary. Numerical experiments in two and three dimensions showcase the validity and efficiency of the estimators.
\end{abstract}



\section{Introduction}
In many engineering applications, partial differential equations (PDEs) must be solved in geometrically complex domains, including geometrical features of different sizes and shapes. The complexity and multiscale nature of such computational geometries make creating an appropriate mesh and subsequent simulations challenging. Even a single, relatively small geometric feature can increase the cost of a simulation by a factor of 10 due to the higher number of degrees of freedom \cite{white_meshing_2003, lee_small_2005}. Therefore, it is common practice to simplify the initial geometry as much as possible by ignoring those features that may not significantly impact the simulation's result. This process is referred to as defeaturing.

Multiple approaches exist for this task based on some a priori criteria, which are applied before any simulation is performed. The latter criteria may be subjective and rely on the engineer's expertise \cite{lee_development_2003}, be based on some geometrical rationale \cite{thakur_survey_2009} or on a priori knowledge of the underlying physical problem \cite{fine_automated_2000,foucault_mechanical_2004,rahimi_cad_2018}.

However, to automate the modeling and simulation process, a posteriori estimators of the defeaturing errors are necessary. In \cite{ferrandes_posteriori_2009}, an a posteriori error estimator is formulated for linear elasticity problems. Although not mathematically certified, it is based on the crucial insight that the defeaturing error must concentrate around the defeatured boundaries. Another line of work referred to as feature sensitivity analysis (FSA) \cite{gopalakrishnan_formal_2007, gopalakrishnan_feature_2008, turevsky_defeaturing_2008, turevsky_efficient_2009} based on topological sensitivity analysis \cite{choi_structural_2005, sokolowski_topological_1999} crucially relies on the assumption of infinitesimally small features and cannot be generalized to more complex situations. In \cite{li_estimating_2011}, an a posteriori estimator for internal holes based on the dual weighted residual (DWR) method \cite{becker_optimal_2001, oden_estimation_2002} is presented with generalizations to different linear and non-linear PDE problems in \cite{li_estimating_2011-1, li_goal-oriented_2013} and \cite{zhang_estimation_2016}. A similar a posteriori error estimator based on the reciprocal theorem, stating that the solution flux is conserved in the features, is presented in \cite{tang_evaluating_2013}. Yet, all the latter works lack a rigorous mathematical certification of the estimators as well as a study of their efficiency. More precisely, the estimators rely on the assumption that the solution on the defeatured geometry is close to the one on the exact geometry and involve heuristic parameters that depend on the size of the feature.

This work is based on the framework introduced in \cite{buffa_analysis-aware_2022}, which formalizes the process of defeaturing for the Poisson equation and provides a mathematically certified and efficient a posteriori estimator for single features subject to Neumann boundary conditions. The estimator is derived from an error representation formula based on the divergence theorem reflecting the intuition in \cite{ferrandes_posteriori_2009} that the defeaturing error is concentrated at the defeatured boundary. In particular, the evaluation of the estimator only requires evaluating boundary integrals over the defeatured boundary. Furthermore, the dependence on the feature size is explicit in contrast to previous works. In \cite{antolin_analysisaware_2024}, the framework is extended to multi-feature geometries, linear elasticity problems and Stokes flow, while \cite{buffa_adaptive_2024} combines the defeaturing estimator with the a posteriori estimator for the numerical error from \cite{buffa_posteriori_2023} to an adaptive meshing procedure, which simultaneously handles the numerical and the defeaturing error in the context of isogeometric analysis (IGA) \cite{hughes_isogeometric_2005, marussig_review_2018}. A comprehensive treatment of these subjects can also found in \cite{chanon_adaptive_2022}. An a posteriori error estimator based on an equilibrated flux reconstruction is proposed in \cite{buffa_equilibrated_2024}, which allows for a sharp bound on the numerical error component and is suitable for finite element discretizations.

We extend the framework in \cite{buffa_analysis-aware_2022} to the case of features with Dirichlet boundary conditions in Poisson problems. We derive a posteriori error estimators for negative features in the interior and the boundary of the computational domain and prove their reliability. Although the error representation formula remains valid in this case, the analysis is more intricate since it involves fractional Sobolev spaces on boundaries with negative exponents. Furthermore, care must be taken when treating mixed boundary conditions. As a model problem, we consider the Poisson equation, but the approach can clearly be generalized to other linear elliptic PDEs.

The paper is structured as follows: \Cref{sec:preliminaries} introduces the functional analytic notions necessary for the mathematical analysis of the defeaturing problem. The model problem is defined in \cref{sec:defeaturing model problems} together with their simplified counterparts. In \cref{sec:estimates}, we define the defeaturing error and the error estimators and prove their reliability. Numerical experiments validating the theory are presented and discussed in \cref{sec:numerical experiments}. Finally, we provide our conclusions in \cref{sec:conclusion and outlook}.

\section{Preliminaries}
\label{sec:preliminaries}
This section provides an overview of the necessary notions from functional analysis that will be used throughout the paper. For clarity, we recall the notation $\lesssim$: by writing $A \lesssim B$, we mean that $A \leq c B$ for some constant $c$ independent of the size of the considered domains. Additionally, we use $A \simeq B$ to indicate that both $A \lesssim B$ and $B \lesssim A$ hold.

\subsection{Domains and boundaries}
We consider $\R^n$ for $n = 2$ or $n = 3$ and start by defining the notion of bounded Lipschitz domain following the definition of \cite{leoni_first_2017}.

\begin{definition}
\label{def:uniformly lipschitz cts boundary}
The boundary $\partial \domain$ of an open bounded set $\domain \subseteq \R^n, n \geq 2$ is \emph{Lipschitz continuous} if there exist $\epsilon, L > 0, M \in \N$, and a locally finite countable open cover $\cover = \{C_k\}_{k \in \N}$ of $\partial\domain$ such that
\begin{enumerate}
    \item if $x \in \partial \domain$, then $\ball{\epsilon}{x} \subset C_k$ for some $C_k \in \cover$,
    \item  no point of $\R^n$ is contained in more than $M$ of the sets in $\cover$,
    \item for each $C_n \in \cover$ there exist local coordinates $y = (y', y_n) \in \R^{n - 1} \times \R$ and a Lipschitz continuous function $f_k: \R^{n- 1} \to \R$ with $\Lip f_k \leq L$, such that $C_k \cap \domain = C_k \cap V_k$, where $V_k$ is given in local coordinates by
    \begin{align*}
        \left\{(y', y_n) \in \R^{n-1}\times \R: y_n > f_k(y')\right\}.
    \end{align*}
\end{enumerate}
The constants $\epsilon$ and $L$ will be referred to as the \emph{covering radius} and the \emph{Lipschitz bound}, respectively.
\end{definition}

\begin{remark}
\label{rem:lipschitz bound and covering radius}
    We underline here the role of the constants $\epsilon$ and $L$ in \cref{def:uniformly lipschitz cts boundary}. They not only encode the regularity but also the relative distances between different parts of the boundary. Indeed, even for a smooth boundary, the covering radius $\epsilon$ might be tiny if the boundary almost self-intersects.
\end{remark}

We subsequently only consider connected Lipschitz domains if not otherwise stated. Let $\genman$ be a $d$-dimensional submanifold of $\R^n$ for $d \in \{n, n-1\}$. If $d = n$, then $|\genman|$ denotes the $n$-dimensional Lebesgue measure of $\genman$, while if $d = n -1$, $|\genman|$ denotes the $(n-1)$-dimensional Hausdorff measure of $\genman$. We denote by $\overline{\genman}$ and $\interior{\genman}$ the closure and interior of $\genman$, respectively. For a non-necessarily connected subset $\genman_*$ of $\genman$ we write $\hull{\genman_*}$ for the convex hull of $\genman_*$ in $\genman$. Furthermore, we denote by $\diam{\genman_*}$ the manifold diameter of $\hull{\genman_*}$, that is
    $\diam{\genman_*} \coloneqq \max\{\rho(x, y) | x, y \in \hull{\genman_*}\}$,
where $\rho(x, y)$ denotes the geodesic distance between two points $x$ and $y$. Finally, we will require the geometrical features, which will be defined in \cref{sec:defeaturing model problems}, to satisfy the following isotropy condition:

\begin{definition}
\label{def:isotropic subset}
    Let $\genman$ be a $d$-dimensional subset of $\R^n, d \in \{n-1, n\}$. We say that $\genman$ is \emph{isotropic} if
    \begin{align*}
        \diam{\genman} \lesssim \max_{\genman_c \in \conn{\genman}} \left(\diam{\genman_c}\right),
    \end{align*}
    and each connected component $\genman_c \in \conn{\genman}$ satisfies $\diam{\genman_c}^d \lesssim |\genman_c|$.
    In particular, if $\genman$ is isotropic and if we let $\genman_{\max} \coloneqq \operatorname*{argmax}_{\genman_c \in \conn{\genman}}\left ( \diam{\genman_c}\right)$, then
    \begin{align*}
        \diam{\genman}^d \lesssim \diam{\genman_{\max}}^d \lesssim |\genman_{\max}| \leq |\genman|.
    \end{align*}
    Moreover, if $\genman$ is connected, we have $\diam{\genman}^d \simeq |\genman|$.
\end{definition}

\subsection{Functional spaces}

Let us introduce the notation for the classical\linebreak Lebesgue and Sobolev spaces that are needed in the subsequent discussion following the reference works \cite{lions_non-homogeneous_1972, grisvard_elliptic_2011, leoni_first_2017} and \cite{leoni_first_2023}.

We denote by $\Lp{\omega}$ the standard Lebesgue spaces of exponent $p \in [1, \infty]$, endowed with the standard norm denoted $||\cdot||_{\Lp{\omega}}$. For a multi-index $\boldsymbol{\alpha} \in \N_0^d$, we denote by $D^{\boldsymbol{\alpha}}$ the partial derivative operator. We denote by $\hs{\omega}$ the standard Sobolev space of order $s \geq 0$, endowed with the standard norm $\norm{\cdot}{s}{\omega}$. Writing $|\boldsymbol{\alpha}| \coloneqq \sum_{i = 1}^d \alpha_i$, the latter is given by
\begin{align}
    \label{eq:definition of fractional sobolev norm}
    \norm{v}{s}{\omega}^2 \coloneqq \norm{v}{\lfloor s \rfloor}{\omega}^2 + \seminorm{v}{\theta}{\omega}^2, \quad \quad
    \norm{v}{\lfloor s \rfloor}{\omega}^2 \coloneqq \sum_{0 \leq |\boldsymbol{\alpha}| \leq \lfloor s \rfloor} \Ltwonorm{D^{\boldsymbol{\alpha}} v}{\omega}^2,
    \\
    \seminorm{v}{\theta}{\omega}^2 \coloneqq \sum_{|\boldsymbol{\alpha}| = \lfloor s \rfloor} \int_\omega \int_\omega \frac{|D^{\boldsymbol{\alpha}}v(x) -  D^{\boldsymbol{\alpha}}v(y)|^2}{|x - y|^{n + 2\theta}} \dd{s}(x) \dd{s}(x),
\end{align}
where $s = \lfloor s \rfloor + \theta \geq 0$ and $\dd{s}$ denotes the $(n-1)$-dimensional Hausdorff measure if $d = n-1$ and the Lebesgue measure if $d = n$.

Similarly, we denote by $H^{-s}(\omega)$ the dual space of $H_0^s(\omega) = \overline{\mathcal{C}_0^\infty(\omega)}$, which is identical to $\hs{\omega}$ for $s \in (0, \tfrac{1}{2}]$ \cite{lions_non-homogeneous_1972}, endowed with the dual norm $\norm{\cdot}{-s}{\omega}$ given by
\begin{align*}
    \norm{v}{-s}{\omega} \coloneqq \sup_{w \in H_0^s(\omega)\setminus\{0\}} \frac{\langle v, w \rangle}{\norm{w}{s}{\omega}}, & & \forall v \in H^{-s}(\omega),
\end{align*}
where $\langle \cdot, \cdot \rangle$ denotes the duality pairing.
Finally, we define the average of a function over $\genman$ by
\begin{align*}
    \overline{v}^{\genman} \coloneqq \frac{1}{|\genman|} \int_{\genman} v(x) \, \dd{s}(x), & & \forall v \in \Ltwo{\genman}.
\end{align*}

\subsection{Trace spaces and subspaces}
\label{subsec:trace spaces}
Trace spaces are essential to deal with boundary conditions.
To that end, let $\domain$ be a Lipschitz domain and $\genbd \subseteq \partial\domain$ with $|\genbd| > 0$ be a subset. We denote by $\htrace{\genbd}$ the Sobolev space on $\genbd$ equipped with the integral norm defined in \cref{eq:definition of fractional sobolev norm}.

In addition, $\traceop_\genbd: \hone{\domain} \to \htrace{\genbd}$ and $\liftop_\genbd :\htrace{\genbd} \to \hone{\domain}$ denote the unique trace and corresponding lifting operator \cite{leoni_first_2017}, respectively, such that\linebreak $\traceop_\genbd(\liftop_\genbd(\mu)) = \mu$ for all $\mu \in \htrace{\genbd}$.
The trace operator allows us to define the space
\begin{align*}
    \honebd{\mu}{\genbd}{\domain} \coloneqq \left \{v \in \hone{\domain} :\traceop_\genbd(v) = \mu\right\},
\end{align*}
for some $\mu \in \htrace{\genbd}$. In particular, it is common to define the \emph{natural norm} in $\htrace{\genbd}$ given by
\begin{align}
\label{eq:def natural norm}
    \htraceinfnorm{\mu}{\partial \domain} \coloneqq \inf_{v \in \honebd{\mu}{\partial \domain}{\domain}} \honenorm{v}{\Omega},
\end{align}
which can be shown to be equivalent to the intrinsic Sobolev norm defined in \cref{eq:definition of fractional sobolev norm} \cite{leoni_first_2017, hsiao_boundary_2021}. This norm equivalence will be important later and we have stated the result for reference in \cref{thm:norm equivalence} in \cref{subsec:norm equivalence}.

In the subsequent discussion, we will require the following subspace of the trace space $\htrace{\genbd}$:
\begin{align}
\label{eq:definition htracedbz}
    \htracedbz{\genbd} \coloneqq \left\{\mu \in \Ltwo{\genbd} : \mu^\star \in \htrace{\partial \domain}\right\},
\end{align}
where $\mu^\star$ denotes the extension of $\mu$ by 0 on $\partial \domain$. The corresponding norm and semi-norm are, respectively, defined by
\begin{align*}
    \htracedbznorm{\mu}{\genbd}^2 &\coloneqq \norm{\mu}{0}{\genbd}^2 + \htracedbzseminorm{\mu}{\genbd}^2,
    \\
    \htracedbzseminorm{\mu}{\genbd}^2 &\coloneqq \htraceseminorm{\mu}{\genbd}^2 + \int_\genbd \int_{\partial \domain \setminus \genbd} \frac{|\mu(y)|^2}{|x - y|^{n}} \dd{s}(x) \dd{s}(y).
\end{align*}
In particular, we have
\begin{align*}
    \htracedbznorm{\mu}{\genbd} = \htracenorm{\mu^\star}{\partial \domain} & &\text{ and } & & \htracedbzseminorm{\mu}{\genbd} = \htraceseminorm{\mu^\star}{\partial\domain}.
\end{align*}

\begin{remark}
    In the treatment of the space $\htracedbz{\Lambda}$, defined in \cref{eq:definition htracedbz}, we must distinguish two cases: If $\partial \genbd = \emptyset$, i.e. if $\Lambda$ is a closed curve or surface, then $\htracedbz{\genbd} = \htrace{\genbd}$ as sets, while their topologies \emph{do not} coincide. In contrast, if $\genbd$ has a boundary itself, i.e. $\partial \genbd \neq \emptyset$, then $\htracedbz{\genbd} \subset \htrace{\genbd}$, where the inclusion is strict in general.
\end{remark}

The space $\htracedbz{\genbd}$ is particularly suitable for analysis, admitting Poincaré- and interpolation-type inequalities, which are stated for reference in \cref{subsec:poincare and interpolation}. Furthermore, due to the norm equivalence between the natural norm in \cref{eq:def natural norm} and the intrinsic Sobolev norm, the former is also closely related to the norm in $\htracedbz{\genbd}$:

\begin{proposition}
\label{prop:norm equivalence for htracedbz}
    Let $\domain \subset \R^n, n \geq 2$ be a domain with Lipschitz boundary $\partial \domain$ and $\genbd \subset \partial \domain$ with $|\genbd| > 0$. Then, for every $\mu \in \htracedbz{\genbd}$,
    \begin{align*}
        \htraceinfnorm{\mu^\star}{\partial \domain} \lesssim |\genbd|^{\frac{-1}{2(n-1)}} \Ltwonorm{\mu}{\genbd} + \htracedbzseminorm{\mu}{\genbd}.
    \end{align*}
\end{proposition}

\begin{proof}
    We argue by a scaling argument, where the rescaled subset $\hat{\genbd}$ of the boundary $\partial \hat{\domain}$ has unit size. Then, consider a covering of the rescaled domain boundary $\partial \hat{\domain}$ with constants $\hat{\epsilon}, \hat{L} > 0$ and $\hat{M} \in \N$ according to \cref{def:uniformly lipschitz cts boundary}, where we may assume without loss of generality that all sets in the covering are balls of radius $\hat{\epsilon}$.

    Rescaling all the balls by the factor $|\genbd|^{\frac{1}{n-1}}$ yields an admissible covering of $\partial \domain$ with $\epsilon = |\genbd|^{\frac{1}{n-1}} \hat{\epsilon}$. Furthermore, the number of overlaps stays constant under rescaling and, therefore, $M = \hat{M}$. Finally, the local Lipschitz constants only depend on the local coordinates. Hence, we have $L = \hat{L}$. Thus, we conclude that the constants in \cref{def:uniformly lipschitz cts boundary} depend on the size of a subset of the boundary as follows:
    \begin{align}
    \label{eq:norm equivalence proof - scaling relations}
        \epsilon \simeq |\genbd|^{\frac{1}{n-1}}, & & L \simeq 1, & & M \simeq 1.
    \end{align}
    Finally, let  $\mu \in \htracedbz{\genbd}$ with $\mu^\star$ the extension by zero to $\htrace{\partial \domain}$. By inserting the scaling relations \cref{eq:norm equivalence proof - scaling relations} into estimate \cref{eq:general norm equivalence} of \cref{thm:norm equivalence}, we find
    $u \in \hone{\domain}$, such that
    \begin{align*}
        \honeseminorm{u}{\domain} \lesssim |\genbd|^{\frac{-1}{2(n-1)}} \Ltwonorm{\mu}{\genbd} + \htraceseminorm{\mu^\star}{\partial \domain},
    \end{align*}
    which concludes the proof.
\end{proof}

\subsection{The Neumann 
trace operator}

We define the Neumann trace operator as the continuous linear operator $\neumop: \hdiv\domain \to \htracedual{\partial \domain}$, such that for $v \in \hdiv\domain$ and $\mu \in \htrace{\domain}$,
\begin{align}
\label{eq:defnition of neumann operator}
    \langle \neumop(v), \mu \rangle \coloneqq \int_\domain v \cdot \nabla \liftop_{\partial \domain}(\mu) \dd{x} + \int_\domain \Div(v) \liftop_{\partial \domain}(\mu) \dd{x},
\end{align}
where $\liftop_{\partial \domain}$ denotes the lifting operator defined in \cref{subsec:trace spaces}.
For subsets of the boundary $\genbd \subset \partial \domain$ with $|\genbd| > 0$, we can naturally extend the definition in \cref{eq:defnition of neumann operator}  to $\partialneumop \genbd: \hdiv\domain \to \htracedbzdual{\genbd}$ by setting
\begin{align}
\label{eq:definition of partial neumann operator}
    \langle \partialneumop\genbd(v), \mu\rangle \coloneqq \langle \neumop(v), \mu^\star \rangle, & & v \in \hdiv\domain,\, \mu \in \htracedbz{\genbd},
\end{align}

In particular, we are interested in the dependence of the operator norm of $\partialneumop\genbd$ on the size of $\genbd$. To that end, let us recall the well-known result for the Neumann trace operator with respect to the natural norm\cite{girault_finite_1986}:
\begin{align}
\label{eq:neumann operator norm}
    \sup_{v \in \hdiv\domain \setminus \{0\}} \sup_{\mu \in \htrace{\partial \domain} \setminus \{0\}} \frac{\langle \neumop(v), \mu \rangle}{\hdivnorm{v}{\domain} \htraceinfnorm{\mu}
    {\partial \domain}} = 1.
\end{align}
The latter does not hold in general for the partial Neumann trace operator $\partialneumop\genbd$.
Nevertheless, if $\genbd$ is not a closed surface, we can establish a result similar to \cref{eq:neumann operator norm} with the help of the norm equivalence in \cref{prop:norm equivalence for htracedbz}:

\begin{lemma}
\label{cor:size independence of neumann trace operator norm}
    Let $\domain \subset \R^n, n \geq 2$ be a Lipschitz domain and $\genbd \subset \partial \domain$ with $|\genbd| > 0$ and $\partial\genbd \neq \emptyset$. Then, it holds for any $v \in \hdiv\domain$ and $\mu \in \htracedbz{\genbd}$ that
    \begin{align*}
        \langle \partialneumop\genbd (v), \mu \rangle \lesssim \hdivnorm{v}{\domain} \htracedbznorm{\mu}{\genbd}.
    \end{align*}
\end{lemma}

\begin{proof}
    Using the definition of the partial Neumann trace operator in \cref{eq:definition of partial neumann operator}, \cref{eq:neumann operator norm} and \cref{prop:norm equivalence for htracedbz}, we find
    \begin{align*}
        \langle \partialneumop\genbd(v), \mu\rangle &\leq \hdivnorm{v}{\domain} \htraceinfnorm{\mu^\star}{\partial \domain} 
        \lesssim \hdivnorm{v}{\domain} \left( |\genbd|^{\frac{-1}{2(n-1)}} \Ltwonorm{\mu}{\genbd} + \htracedbzseminorm{\mu}{\genbd} \right).
    \end{align*}
     Applying the Poincaré inequality from \cref{lemma:poincare I} to $\Ltwonorm{\mu}{\genbd}$, we conclude that
    \begin{align*}
        \langle \partialneumop\genbd(v), \mu\rangle &\leq \hdivnorm{v}{\domain} \htraceinfnorm{\mu^\star}{\partial \domain} 
        \lesssim \hdivnorm{v}{\domain} \htracedbznorm{\mu}{\genbd}.
    \end{align*}
\end{proof}

\begin{remark}
\label{rem:size independence for closed surfaces}
    Note that the argument in the proof of \cref{cor:size independence of neumann trace operator norm} fails for closed surfaces $\genbd$ since in this case the space $\htracedbz{\genbd}$ contains constants and, therefore, the Poincaré inequality does not apply in this form. However, for functions with zero average, the same argument applies with \cref{lemma:poincare II}.
\end{remark}

\section{Defeaturing model problems}
\label{sec:defeaturing model problems}
In this section, we formulate the defeaturing problem for negative features equipped with Dirichlet boundary conditions following the framework introduced in \cite{buffa_analysis-aware_2022}. For ease of exposition, we restrict ourselves to geometries with a single feature. The subsequent results generalize readily to multiple negative features as seen in \cite{buffa_analysis-aware_2022, antolin_analysisaware_2024, buffa_adaptive_2024, buffa_equilibrated_2024}.

\begin{definition}
\label{def:negative feature domain}
    Let $\domain \subset \R^n, n \in \{2, 3\}$ be a Lipschitz domain. We say that $\domain$ has a \emph{negative feature} if it is of the form $\domain = \defdomain \setminus \overline{\feat}$
    for a Lipschitz domain $\defdomain$ and an isotropic Lipschitz domain $\feat$ according to \cref{def:isotropic subset} such that $\defdomain \cap \feat \neq \emptyset$.
    We say that the feature $\feat$ is \emph{internal} if $\overline{F} \subset \defdomain$. Otherwise, we say that it is a \emph{boundary feature}.

    We define the \emph{defeatured boundary} by $\defbd \coloneqq \partial \feat \cap \partial \domain$. The \emph{simplified feature boundary} is defined as $\simpbd \coloneqq \partial \feat \setminus \overline{\defbd}$ for boundary features. The \emph{remaining boundary} is defined as $\rmbd \coloneqq \partial \domain \setminus \overline{\defbd}$. We split the remaining boundary into the remaining Dirichlet and Neumann boundary, $\rmbd_D$ and $\rmbd_N$, respectively, such that $\rmbd_D \cap \rmbd_N = \emptyset$ and $\rmbd_D \cup \rmbd_N = \rmbd$.
\end{definition}

\begin{figure}
    \centering
    \begin{tikzpicture}
    \draw[thin, gray, dashed] (-4,0) rectangle (4, 4);
    \node[black] at (0, 2) {$\domain$};

    \draw[very thick, gray] (2,2.5) circle [radius=0.75];
    \node[black] at (2,2.5) {$\feat^{(\mathrm{I})}$};
    \node[black] at (3.0, 2) {$\defbd^{(\mathrm{I})}$};

    \draw[very thick, gray] (-1,0) arc[start angle=180, end angle=0, radius=0.75];
    \node[black] at (-0.25, 0.3) {$\feat^{(\mathrm{DN})}$};
    \node[black] at (0.8, 0.75) {$\defbd^{(\mathrm{DN})}$};
    \node[black] at (-0.25, -0.5) {$\simpbd^{(\mathrm{DN})}$};

    \draw[very thick, green1] (-4, 1) -- (-4, 0) -- (-2, 0);
    \node[green1] at (-3, -0.5) {$\Tilde{\rmbd}_D$};

    \draw[very thick, gray] (-4, 3) -- (-4, 4) -- (-3.5, 4);
    \draw[very thick, gray] (-2, 4) -- (4, 4) -- (4, 0) -- (2, 0);
    \node[anchor=east, black] at (5.5, 2) {$\rmbd_D^{(\mathrm{DD})}$};
    
    \draw[thick, blue1] (-2,-0.1) -- (-2,0.1);
    \draw[thick, blue1] (2,-0.1) -- (2,0.1);
    \draw[thick, blue1] (-1,-0.1) -- (-1,0.1);
    \draw[thick, blue1] (0.5,-0.1) -- (0.5,0.1);
    \draw[very thick, blue1] (-2, 0) -- (-1, 0);
    \draw[very thick, blue1] (0.5, 0) -- (2, 0);
    \node[blue1] at (1.5, -0.5) {$\rmbd_N^{(\mathrm{DN})}$};

    \draw[thick, red1] (-4.1,1) -- (-3.9,1);
    \draw[thick, red1] (-4.1,3) -- (-3.9,3);
    \draw[very thick, red1] (-4, 1) -- (-4, 3);
    \node[red1, anchor=west] at (-5, 2) {$\Tilde{\rmbd}_N$};

    \draw[very thick, gray] (-2,4) arc[start angle=0, end angle=-180, radius=0.75];
    \node[black] at (-2.75, 3.70) {$\feat^{(\mathrm{DD})}$};
    \node[black] at (-1.75, 3.2) {$\defbd^{(\mathrm{DD})}$};
    \node[black] at (-2.375, 4.5) {$\simpbd^{(\mathrm{DD})}$};
\end{tikzpicture}
    \caption{Example of a two-dimensional domain with an internal negative feature $\feat^{(\mathrm{I})}$, a Dirichlet-Dirichlet boundary feature $\feat^{(\mathrm{DD})}$, and a Dirichlet-Neumann boundary feature $\feat^{(\mathrm{DN})}$. The remaining boundary $\rmbd$ is given by $\rmbd = \rmbd_D \cup \rmbd_N$, where the remaining Dirichlet boundary $\rmbd_D$ consists of two connected components, namely $\rmbd_D^{(\mathrm{DD})}$, connected to the corresponding boundary feature, and $\Tilde{\rmbd}_D$, not touching any feature boundary. The remaining Neumann boundary $\rmbd_N$ has the analogous structure $\rmbd_N = \rmbd_N^{(\mathrm{DN})} \cup \Tilde{\rmbd}_N$. The dashed lines represent the simplified boundaries $\simpbd^{(\mathrm{DD})}$ and $\simpbd^{(\mathrm{DD})}$ for the Dirichlet-Dirichlet and Dirichlet-Neumann features, respectively.}
    \label{fig:features}
\end{figure}
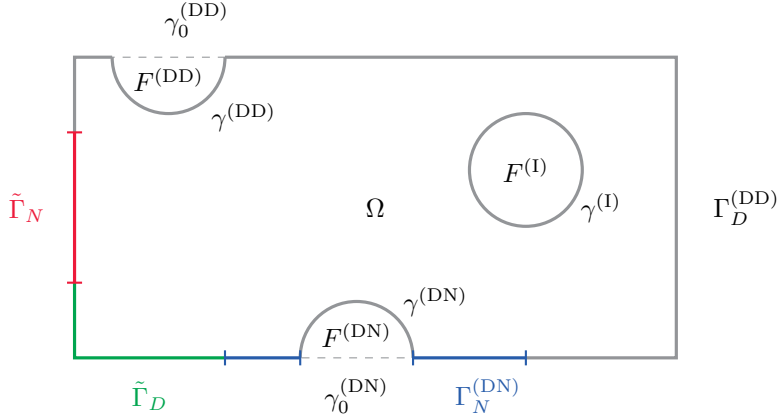

We now formulate the generic Poisson problem we want to solve in a domain $\domain$ with a negative feature $F$ and its corresponding defeatured counterpart $\defdomain$. We start with the problem in the exact geometry $\domain$ containing the feature:

\begin{problem}[Exact problem]
\label{problem:exact problem}
    Let $\domain = \defdomain \setminus\overline{\feat}$ be a Lipschitz domain with a negative feature $\feat$ and feature boundary $\defbd$ according to \cref{def:negative feature domain}. Then, solve
    \begin{nalign}
    \label{eq:exact poisson problem}
    \text{Find } u: \domain \to \R, \text{ solution of}\\
    \begin{cases}
        -\Delta u = f,& \text{ in } \domain,\\
        u = g_\defbd,& \text{ on } \defbd,\\
        u = g_D,& \text{ on } \rmbd_D,\\
        \partial_n u = g_N,& \text{ on } \rmbd_N,
    \end{cases}
\end{nalign}
for $f \in \Ltwo{\domain}, g_\defbd \in \htrace{\defbd}, g_D \in \htrace{\rmbd_D}$, and $g_N \in \Ltwo{\rmbd_N}$, such that $g_\defbd$ and $g_D$ are compatible whenever $\overline{\defbd} \cap \overline{\rmbd_D} \neq \emptyset$.
\end{problem}

The formulation of the defeatured problem depends on the geometry of $\domain$ and the intersections $\overline{\defbd} \cap \overline{\rmbd}_D$ and $\overline{\defbd} \cap \overline{\rmbd}_N$, which are illustrated in \cref{fig:features}:

\begin{problem}[Defeatured problem]
\label{problem:defeatured problem}
\problemequations
Let $\domain = \defdomain \setminus\overline{\feat}$ be a Lipschitz domain with a negative feature $\feat$, feature boundary $\defbd$, remaining Dirichlet boundary $\rmbd_D$, and remaining Neumann boundary $\rmbd_N$ according to \cref{def:negative feature domain}. We define the following defeatured problems, depending on the geometry of $\domain$:
\begin{enumerate}
    \item Dirichlet-Dirichlet case: Consider a boundary feature $F$ such that $\overline{\rmbd}_D \cap \overline{\gamma} \neq \emptyset$ and $\overline{\rmbd}_N \cap \overline{\gamma} = \emptyset$. That is, the feature only touches the remaining Dirichlet boundary. Then, the defeatured problem reads
    \begin{nalign}
    \label{eq:ddd defeatured problem}
    \text{Find } u_0: \defdomain \to \R, \text{ solution of}\\
    \begin{cases}
        -\Delta u_0 = \Tilde{f},& \text{ in } \defdomain,\\
        u_0 = \Tilde{g}_D,& \text{ on } \rmbd_D \cup \simpbd,\\
        \partial_n u_0 = g_N,& \text{ on } \rmbd_N,
    \end{cases}
    \end{nalign}
    where $\Tilde{f} \in \Ltwo{\defdomain}$ is an extension of $f$, and $\Tilde{g}_D \in \htrace{\rmbd_D \cup \simpbd}$ is an extension of $g_D$.

    \item Dirichlet-Neumann case: Consider a boundary feature $F$ such that $\overline{\rmbd}_N \cap \overline{\gamma} \neq \emptyset$ and $\overline{\rmbd}_D \cap \overline{\gamma} = \emptyset$. That is, the feature only touches the remaining Neumann boundary. Then, the defeatured problem reads
    \begin{nalign}
    \label{eq:ndn defeatured problem}
    \text{Find } u_0: \defdomain \to \R, \text{ solution of}\\
    \begin{cases}
        -\Delta u_0 = \Tilde{f},& \text{ in } \defdomain,\\
        u_0 = g_D,& \text{ on } \rmbd_D,\\
        \partial_n u_0 = \Tilde{g}_N,& \text{ on } \rmbd_N \cup \simpbd,
    \end{cases}
    \end{nalign}
    where $\Tilde{f} \in \Ltwo{\defdomain}$ is an extension of $f$, and $\Tilde{g}_N \in \Ltwo{\rmbd_D \cup \simpbd}$ is an extension of $g_N$.

    \item Internal case: $\feat$ is an internal feature. Then, the defeatured problem reads
        \begin{nalign}
    \label{eq:internal defeatured problem}
    \text{Find } u_0: \defdomain \to \R, \text{ solution of}\\
    \begin{cases}
        -\Delta u_0 = \Tilde{f},& \text{ in } \defdomain,\\
        u_0 = g_D,& \text{ on } \rmbd_D,\\
        \partial_n u_0 = g_N,& \text{ on } \rmbd_N,
    \end{cases}
    \end{nalign}
    where $\Tilde{f} \in \Ltwo{\defdomain}$ is an extension of $f$.
\end{enumerate}
\end{problem}
Several remarks are in order. First, we stress that we do not impose a specific extension of the source term $f$. However, it is clear that this choice significantly influences the defeaturing error. We will return to this in \cref{sec:numerical experiments}. 

Second, we justify our choice of imposing Neumann boundary conditions on the simplified feature boundary of Dirichlet-Neumann features by the observation that when discretizing the defeatured problem, imposing Dirichlet boundary conditions on the simplified boundary would require us to resolve the intersection between the outer Neumann boundary $\rmbd_N$ and the simplified feature boundary $\defbd_0$. The latter may also be hard to mesh, hence we want to avoid it.

Finally, we note that the assumption that Dirichlet-Neumann boundary features only touch the outer Neumann boundary can be relaxed. The following arguments demonstrate that they may also be connected to the outer Dirichlet boundary; however, we maintain this assumption for ease of exposition.

 \section{Defeaturing error estimates}
\label{sec:estimates}
In this section, we discuss the error caused by solving one of the equations \cref{eq:ddd defeatured problem,eq:ndn defeatured problem,eq:internal defeatured problem} of \cref{problem:defeatured problem} instead of \cref{problem:exact problem}. We start by defining the defeaturing error and presenting an error representation formula on the feature boundary before we derive reliable error indicators for each of the three cases in \cref{problem:defeatured problem}. 

\subsection{Defeaturing error and error representation}
Consider the situation in \cref{sec:defeaturing model problems} and let $u$ and $u_0$ be solutions to \cref{problem:exact problem} and one of the variants of \cref{problem:defeatured problem}, respectively. Then, we define the defeaturing error $e \in \honebd{0}{\rmbd_D}{\domain}$ by $e \coloneqq u - (u_0)_{|\domain}$.
Note that in all three cases of \cref{problem:defeatured problem}, the defeaturing error satisfies the PDE problem
\begin{align}
\label{eq:error pde}
    \begin{cases}
        -\Delta e = 0,& \text{ in } \domain,\\
        e = d_\defbd,& \text{ on } \defbd,\\
        e = 0,& \text{ on } \rmbd_D,\\
        \partial_n e = 0,& \text{ on } \rmbd_N,
    \end{cases}
\end{align}
where $\bderr \coloneqq g_\defbd - (u_0)_{|\defbd}$ denotes the error on the defeatured boundary.

We derive an error representation formula based on the PDE problem satisfied by the defeaturing error function. Let us first consider a case where no Neumann boundaries are present in the problem; that is, $\rmbd = \rmbd_D$. Then, it is clear that $\bderr \in\htracedbz{\defbd}$ and $e_{|\rmbd} \equiv 0 \in \htracedbz{\rmbd}$. Hence, we may apply the divergence theorem \cite{girault_finite_1986}[Corollary 2.6] to obtain
\begin{align}
\label{eq:dd and internal error representation formula}
    \honeseminorm{e}{\domain}^2 = \langle \neumop(\nabla e), e_{|\partial \domain} \rangle = \langle 
    \partialneumop{\rmbd}(\nabla e), \bderr \rangle +
    \langle \partialneumop{\rmbd}(\nabla e), e_{|\rmbd} \rangle = \langle 
    \partialneumop{\defbd}(\nabla e), \bderr \rangle.
\end{align}
In the presence of outer Neumann boundaries, which are not connected to the feature, see e.g. $\Tilde{\rmbd}_N$ in \cref{fig:features}, it is clear that $e_{|\Tilde{\rmbd}_N} \in \htracedbz{\Tilde{\rmbd}_N}$ and $e_{|\rmbd_D \cup \defbd} \in \htracedbz{\rmbd_D \cup \defbd}$ because $e$ vanishes on $\rmbd_D$. Since $\partialneumop{\Tilde{\rmbd}_N}(\nabla e) \equiv 0$, it follows that
\begin{align*}
    \langle \neumop(\nabla e), e_{|\partial \domain} \rangle = \langle 
    \partialneumop{\Tilde{\rmbd}_N}(\nabla e), e_{|\Tilde{\rmbd}_N} \rangle +
    \langle \partialneumop{\rmbd_D \cup \defbd}(\nabla e), e_{|\rmbd_D \cup \defbd} \rangle = \langle \partialneumop{\rmbd_D \cup \defbd}(\nabla e), e_{|\rmbd_D \cup \defbd} \rangle.
\end{align*}
Hence, the error representation formula in \cref{eq:dd and internal error representation formula} remains valid in this case.

In contrast, in the case of a Dirichlet-Neumann feature, $\bderr$ is in general not an element of $\htracedbz{\defbd}$, but merely of $\htrace{\defbd}$. Indeed, the defeaturing error need not vanish at the intersection $\overline{\defbd} \cap \overline{\rmbd_N^{(\mathrm{DN})}}$; see \cref{fig:features}. However, consider an arbitrary extension $\Tilde{d}_\defbd$ of $\bderr$ to $\htracedbz{\defbd \cup \rmbd_N^{(\mathrm{DN})}}$. Since $\neumop(\nabla e)$ vanishes on $\rmbd_N$, it holds that
\begin{align}
\label{eq:dn error representation formula}
    \honeseminorm{e}{\domain}^2 = \langle \partialneumop{\defbd}(\nabla e), e_{|\defbd \cup \rmbd_N} \rangle = \langle \partialneumop{\defbd}(\nabla e), \Tilde{d}_\gamma \rangle.
\end{align}
Therefore, once we have derived an error estimate for the Dirichlet-Dirichlet case from \cref{eq:dd and internal error representation formula}, we can reuse that argument for the Dirichlet-Neumann case if we are able to construct a well-behaved extension $\Tilde{d}_\defbd$ of the boundary error $\bderr$. Such a continuous extension operator is constructed in \cite{weder_extension_2025} and we have re-stated the result for completeness in \cref{subsec:extensions}[\cref{thm:boundary extension operator}].

We now derive reliable defeaturing error estimates for the three cases \cref{eq:ddd defeatured problem,eq:ndn defeatured problem,eq:internal defeatured problem} of \cref{problem:defeatured problem} based on the boundary representation formulas \cref{eq:dd and internal error representation formula,eq:dn error representation formula} and Sobolev estimates in trace spaces.

\subsection{Negative boundary features - Dirichlet-Dirichlet case}
In the\linebreak Dirichlet-Dirichlet case, we define the defeaturing estimator by
\begin{align}
\label{eq:dirichlet dirichlet estimator}
    \defestdd{u_0} \coloneqq \sqrt{2 \Ltwonorm{\bderr}{\defbd} \Ltwonorm{\nabla_t \bderr}{\defbd}},
\end{align}
where $\nabla_t \bderr$ denotes the tangential gradient of $\bderr$ along $\defbd$.

\begin{theorem}
\label{thm:dirichlet dirichlet reliability}
    The defeaturing error estimator defined in \cref{eq:dirichlet dirichlet estimator} is reliable for Dirichlet-Dirichlet features. That is, if $\overline{\defbd} \cap \overline{\rmbd_D} \neq \emptyset$ and $\overline{\defbd} \cap \overline{\rmbd_N} = \emptyset
    $, then
    \begin{align*}
        \honeseminorm{e}{\domain} \lesssim \defestdd{u_0}.
    \end{align*}
\end{theorem}

\begin{proof}
    Using the error representation formula \cref{eq:dd and internal error representation formula} and \cref{cor:size independence of neumann trace operator norm}, we immediately find
    \begin{align*}
        \honeseminorm{e}{\domain}^2 =  \langle \partialneumop{\defbd}(\nabla e), \bderr \rangle \lesssim \honeseminorm{e}{\domain} \htracedbznorm{\bderr}{\defbd}.
    \end{align*}
    Moreover, using the Poincaré and interpolation inequalities in \cref{lemma:poincare I} and\linebreak \cref{eq:interpolation and poincare inequality}, respectively, we find that
    \begin{align*}
        \htracedbznorm{\bderr}{\defbd}^2 \lesssim \left(1 + |\defbd|^{\frac{1}{n-1}}\right)\htraceseminorm{\bderr^\star}{\partial \domain}^2 \lesssim 2 \Ltwonorm{\bderr}{\defbd} \Ltwonorm{\nabla_t \bderr}{\defbd},
    \end{align*}
    which finishes the proof.
\end{proof}

\subsection{Negative boundary features - Dirichlet-Neumann case}
In the\linebreak Dirichlet-Neumann case, we define the defeaturing estimator by
\begin{align}
\label{eq:dirichlet neumann estimator}
    \defestdn{u_0} \coloneqq \defestdnavg{u_0} + \defestdnnavg{u_0},
\end{align}
where we define the average and non-average components, respectively, by
\begin{align*}
    \defestdnavg{u_0} \coloneqq |\defbd|^{\frac{n-2}{2(n-1)}}|\avg{\bderr}{\defbd}| && \text{ and } && \defestdnnavg{u_0} \coloneqq \sqrt{2\Ltwonorm{\bderr - \avg{\bderr}{\defbd}}{\defbd} \Ltwonorm{\nabla_t \bderr}{\defbd}}.
\end{align*}

\begin{theorem}
\label{thm:dirichlet neumann reliability}
    The defeaturing error estimator defined in \cref{eq:dirichlet neumann estimator} is reliable for Dirichlet-Neumann features. That is, if $\overline{\defbd} \cap \overline{\rmbd_N} \neq \emptyset$ and $\overline{\defbd} \cap \overline{\rmbd_D} = \emptyset$, then
    \begin{align*}
        \honeseminorm{e}{\domain} \lesssim \defestdn{u_0}.
    \end{align*}
\end{theorem}

\begin{proof}
First, we point out that in the Dirichlet-Neumann case, $\bderr$ is in general not an element of $\htracedbz{\defbd}$. For ease of notation, we define $\genbd \coloneqq \defbd \cup \rmbd_N$. We now extend $\bderr$ to $\genbd$ using the extension operator from \cref{thm:boundary extension operator}. By inserting the extension into the representation formula \cref{eq:dn error representation formula}, we obtain from Cauchy-Schwarz and \cref{cor:size independence of neumann trace operator norm} that
\begin{align*}
    \honeseminorm{e}{\domain}^2 &= \langle \partialneumop{\genbd}(\nabla e), \extop{\defbd}{\genbd}(\bderr - \avg{\bderr}{\defbd}) + \avg{\bderr}{\defbd} \extop{\defbd}{\genbd}(1) \rangle\\
    &\lesssim \honeseminorm{e}{\domain} \left(\htracedbznorm{\extop{\defbd}{\genbd}(\bderr - \avg{\bderr}{\defbd})}{\genbd} + |\avg{\bderr}{\defbd}| \htracedbznorm{\extop{\defbd}{\genbd}(1)}{\genbd}\right),
\end{align*}
where $1 \in \htrace{\defbd}$ denotes the constant function.
The continuity property of $\extop{\defbd}{\genbd}$ in \cref{thm:boundary extension operator} and the Poincaré inequality in \cref{lemma:poincare II} yield the estimate
\begin{align}
    \honeseminorm{e}{\domain}^2 &\lesssim |\defbd|^{\frac{-1}{n-1}} \Ltwonorm{\bderr - \avg{\bderr}{\defbd}}{\defbd}^2 + \htraceseminorm{\bderr - \avg{\bderr}{\defbd}}{\defbd}^2 + |\defbd|^{\frac{-1}{n-1}} |\avg{\bderr}{\defbd}|^2 \Ltwonorm{1}{\defbd}^2\nonumber
    \\
    \label{eq:dn htrace norm estimate}
    &\lesssim 2 \htraceseminorm{\bderr - \avg{\bderr}{\defbd}}{\defbd}^2 + |\defbd|^{\frac{n-2}{n-1}} |\avg{\bderr}{\defbd}|^2.
\end{align}
By applying the interpolation estimate from \cref{eq:interpolation and poincare inequality} to the first term in \cref{eq:dn htrace norm estimate}, we find after taking the square root on both sides that
\begin{align*}
    \honeseminorm{e}{\domain} \lesssim \sqrt{2 \Ltwonorm{\bderr - \avg{\bderr}{\defbd}}{\defbd} \Ltwonorm{\nabla_t \bderr}
    {\defbd}} + |\defbd|^{\frac{n-2}{2(n-1)}} |\avg{\bderr}{\defbd}|,
\end{align*}
as desired.
\end{proof}

\subsection{Internal negative features}
We cannot apply the previous arguments to internal features since \cref{cor:size independence of neumann trace operator norm} does not hold; however, we can treat the average of $\bderr$ separately, as noted in \cref{rem:size independence for closed surfaces}.
Additionally, it can be shown that \cref{prop:norm equivalence for htracedbz} yields only a suboptimal estimate for constant functions on internal features. The following estimate is sharp instead, assuming no further knowledge of the geometry:

\begin{lemma}
\label{lemma:natural norm estimate for constants}
Let $\domain = \defdomain \setminus \feat \subset \R^n, n \geq 2$ be a Lipschitz domain with internal Lipschitz feature $\feat$ and feature boundary $\defbd = \partial \feat$. We write $\rmbd = \partial \defdomain$ and we denote by $m_\feat$ the barycenter of $\feat$.

Then, it holds for $1 \in \htracedbz{\defbd}$ that,
\begin{align*}
    \htraceinfnorm{1^\star}{\partial \domain} \lesssim \Bar{c}_\defbd,
\end{align*}
where
\begin{align*}
    \Bar{c}_\defbd \coloneqq  
        \begin{cases}
            \sqrt{\frac{2 \pi}{|\log\left(s_\defbd\right)|}}, & n = 2,\\
            \sqrt{\frac{2 \pi \,\diam{\defbd}}{1 - s_\defbd}}, & n=3
        \end{cases}, && \text{ and } && s_\defbd \coloneqq \frac{\diam{\defbd}}{2 \dist(m_F, \rmbd)}.
\end{align*}
\end{lemma}

\begin{proof}
For ease of notation, we define $r \coloneqq \diam{\defbd} / 2$ and $R \coloneqq \dist(m_F, \rmbd)$ such that $s_\defbd = r / R$. Furthermore, assume without loss of generality that $m_\feat = 0$. Then, we have in particular that 
\begin{align*}
    \overline{F} \subset \ball{r}{0} \subset \ball{R}{0} \subset \defdomain.
\end{align*}
By definition of the natural norm, every function $v \in \hone{\domain}$ such that $v_{|\partial \domain} = 1^\star \in \htrace{\partial \domain}$ satisfies
\begin{align*}
    \htraceinfnorm{1^\star}{\partial \domain} \leq \honeseminorm{v}{\domain}.
\end{align*}
Consider the function $v_0 \in \hone{\ball{R}{0} \setminus \ball{r}{0}}$ solution to the PDE problem
\begin{align*}
    \begin{cases}
        -\Delta v_0 = 0, \text{ in } \ball{R}{0} \setminus \overline{\ball{r}{0}},\\
        v_0 = 1, \text{ on } \partial \ball{r}{0}\\
        v_0 = 0, \text{ on } \partial \ball{R}{0},
    \end{cases}
\end{align*}
and define a lifting $v \in \hone{\domain}$ of $1^\star \in \htrace{\partial \domain}$ by setting
\begin{align*}
v \coloneqq \ind{\ball{r}{0} \setminus \overline{\feat}} + \ind{\ball{R}{0} \setminus \overline{\ball{r}{0}}} v_0,
&& \text{such that} &&
    \htraceinfnorm{1^\star}{\partial \domain} \leq \honeseminorm{v}{\domain} = \honeseminorm{v_0}{\ball{R}{0} \setminus \overline{\ball{r}{0}}}.
\end{align*}
A short computation shows that
\begin{align*}
\honeseminorm{v_0}{\domain} \simeq
    \begin{cases}
        \sqrt{\frac{2 \pi}{|\log(r / R)|}},& n=2\\
        \sqrt{\frac{4 \pi r}{1 - r/R}},& n=3,
    \end{cases}
\end{align*}
which concludes the proof.
\end{proof}

We now define the defeaturing error estimator for internal features by
\begin{align}
\label{eq:internal estimator}
    \defestint{u_0} \coloneqq  \defestintavg{u_0} + \defestintnavg{u_0},
\end{align}
where the average and non-average components are, respectively, defined by
\begin{align*}
    \defestintavg{u_0} \coloneqq \Bar{c}_\defbd |\avg{\bderr}{\defbd}|, & & \text{ and } & & \defestintnavg{u_0} \coloneqq 2 \sqrt{\Ltwonorm{\bderr - \avg{\bderr}{\defbd}}{\defbd} \Ltwonorm{\nabla_t \bderr}{\defbd}},
\end{align*}
and $\Bar{c}_\defbd$ is defined in \cref{lemma:natural norm estimate for constants}.

\begin{theorem}
\label{thm:internal estimate reliability}
    The defeaturing error estimator defined in \cref{eq:dirichlet neumann estimator} is reliable for internal features. That is, if $\overline{\defbd} \cap \overline{\partial \domain} = \emptyset$, then
    \begin{align*}
        \honeseminorm{e}{\domain} \lesssim \defestint{u_0}.
    \end{align*}
\end{theorem}

\begin{proof}
    Since the defeaturing error $e \in \hone{\domain}$ is the unique weak solution to problem \cref{eq:error pde}, it is the unique minimizer of the $H^1$-norm for the boundary condition $(\bderr)^\star \in \htrace{\partial \domain}$. Therefore, we have
    \begin{align*}
        \honeseminorm{e}{\domain} = \htraceinfnorm{(\bderr)^\star}{\partial \domain} \leq \htraceinfnorm{(\bderr - 
        \avg{\bderr}{\defbd})^\star}{\partial \domain} + |\avg{\bderr}{\defbd}|\htraceinfnorm{1^\star}{\partial \domain}.
    \end{align*}
    Using \cref{lemma:natural norm estimate for constants}, we find for the second term,
    \begin{align*}
        |\avg{\bderr}{\defbd}|\htraceinfnorm{1^\star}{\partial \domain} \lesssim \Bar{c}_\defbd |\avg{\bderr}{\defbd}|.
    \end{align*}
    Combining \cref{prop:norm equivalence for htracedbz} and \cref{lemma:poincare II} yields for the first term,
    \begin{align}
    \label{eq:internal features natural norm estimate}
        \htraceinfnorm{(\bderr - 
        \avg{\bderr}{\defbd})^\star}{\partial \domain} &\lesssim |\defbd|^{\frac{-1}{2(n-1)}} \Ltwonorm{\bderr - \avg{\bderr}{\defbd}}{\defbd} + \htracedbzseminorm{\bderr - \avg{\bderr}{\defbd}}{\defbd}\nonumber\\
        &\lesssim \htraceseminorm{\bderr - \avg{\bderr}{\defbd}}{\defbd} + \htracedbzseminorm{\bderr - \avg{\bderr}{\defbd}}{\defbd}.
    \end{align}
    Next, we note that in the case of an interior feature, we have $\dist(\partial \domain, \defbd) > 0$, and therefore
    \begin{nalign}
    \label{eq:interior feature norm equivalence}
        \htracedbzseminorm{\bderr - \avg{\bderr}{\defbd}}{\defbd}^2 = \htraceseminorm{\bderr - \avg{\bderr}{\defbd}}{\defbd}^2 + \int_{\partial \domain \setminus \defbd} \int_{\defbd} \frac{|\bderr(x) - \avg{\bderr}{\defbd}|^2}{|x - y|^n}\dd{s(x)}\dd{s(y)}
        \\
        \lesssim \htraceseminorm{\bderr - \avg{\bderr}{\defbd}}{\defbd}^2 + \frac{|\partial \domain \setminus \defbd|}{\dist(\partial \domain \setminus \defbd, \defbd)^n} \Ltwonorm{\bderr - \avg{\bderr}{\defbd}}{\defbd}^2
        \\
        \lesssim \htraceseminorm{\bderr - \avg{\bderr}{\defbd}}{\defbd}^2,
    \end{nalign}
    where we have applied \cref{lemma:poincare II} in the last inequality and used the fact that $\dist(\partial \domain \setminus \defbd, \defbd) \simeq |\defbd|^{\frac{1}{n-1}}$ and that the measure of $\partial \domain \setminus \defbd$ is proportional to the one of $\defbd$ under rescaling.
    By inserting the estimate \cref{eq:interior feature norm equivalence} into \cref{eq:internal features natural norm estimate} and applying \cref{lemma:poincare II} together with \cref{eq:interpolation and poincare inequality}, we finally obtain
    \begin{align*}
        \htraceinfnorm{(\bderr - 
        \avg{\bderr}{\defbd})^\star}{\partial \domain} \lesssim 2 \htraceseminorm{\bderr - \avg{\bderr}{\defbd}}{\defbd} \lesssim
        2 \sqrt{\Ltwonorm{\bderr - \avg{\bderr}{\defbd}}{\defbd} \Ltwonorm{\nabla_t \bderr}{\defbd}},
    \end{align*}
    which concludes the proof.
\end{proof}

\subsection{Extension to multi-feature geometries}
Given the structure of the estimators \cref{eq:dirichlet dirichlet estimator,eq:dirichlet neumann estimator,eq:internal estimator}, nothing prevents us from extending the present results to multi-feature geometries in the same way as \cite{antolin_analysisaware_2024, buffa_adaptive_2024, chanon_adaptive_2022}; that is, the multi-feature estimator is given by the square root of the squared estimators for each individual feature. However, note that the constant $\Bar{c}_\defbd$ in the estimator for internal features \cref{eq:internal estimator} then also depends on the distance to all the other features. In applications, the latter information must be retained after the defeaturing step together with the defeatured boundary information.

\section{Numerical experiments}
\label{sec:numerical experiments}
In this section, we present numerical experiments validating the theory developed above. The computational geometries and meshes in two and three dimensions were created with the help of \texttt{gmsh} \cite{geuzaine_gmsh_2009}. The meshes are, by default, conformal across lower dimensional entities in the geometry. Hence, the exact and defeatured geometries could be meshed simultaneously. The relevant sub-meshes and degrees of freedom were then extracted at the simulation stage, making the comparison between the exact and defeatured solutions straightforward. The finite element simulations were implemented using the \texttt{FEniCSx} library \cite{baratta_dolfinx_2023}. Moreover, fine meshes with local refinement around the features were used so that numerical errors could be neglected. Nonetheless, we stress that the estimators introduced above are independent of the discretization method and the geometry, provided that one can numerically integrate along the feature boundary.

When assessing the performance of the estimators, we will be especially interested in the \emph{effectivity index} $\esteff$ defined by
\begin{align*}
    \esteff \coloneqq \frac{\mathcal{E}(u_0)}{\honeseminorm{e}{\domain}},
\end{align*}
where $\mathcal{E}(u_0)$ denotes one of the defeaturing estimators \cref{eq:dirichlet dirichlet estimator,eq:dirichlet neumann estimator,eq:internal estimator}.
We aim for this constant to be close to one for any feature. Based on \cref{thm:dirichlet dirichlet reliability,thm:dirichlet neumann reliability,thm:internal estimate reliability}, we expect this quantity to remain approximately constant across different feature sizes. Furthermore, it reflects the dependency of the hidden constant in the estimate on quantities, which are not explicitly tracked by the estimator, such as the feature shape.

\subsection{Dirichlet-Dirichlet features}
We start by verifying the reliability of the estimator for Dirichlet-Dirichlet features with respect to the feature's size and shape in two and three dimensions. In addition, we will demonstrate the estimator's ability to capture the relationship between the geometry and the physics by considering two different types of boundary conditions: We call a combination of feature boundary condition $g_\defbd$ and simplified boundary condition $\Tilde{g}_D$ \emph{regular}, if the solution of \cref{eq:exact poisson problem} converges to the solution of \cref{eq:ddd defeatured problem} in the limit $|\defbd| \to 0$, and \emph{singular} otherwise. In other words, if we impose a singular boundary condition on a feature, its presence introduces a singular perturbation to the defeatured problem. Consequently, we expect the defeaturing error to remain non-negligible regardless of the feature size, while for regular boundary conditions, the error vanishes in the limit $|\defbd| \to 0$. An example of a singular boundary condition is provided in the following experiment in two dimensions, while the regular case is treated in the three-dimensional one.

\paragraph{Two-dimensional geometries}
Consider the following setup: The defeatured domain is the unit square $\defdomain = [-\frac 1 2, \frac 1 2]^2 \subset \R^2$ with the negative features cut out from the top side, as illustrated in \cref{fig:boundary features}. For the triangle feature, we consider the opening angle $\alpha = 15\degree$.
We impose a constant source term $f \equiv 1$ in $\domain$, extended to $\Tilde{f} \equiv 1$ in $\defdomain$. The Dirichlet data on the feature is set to
\begin{align*}
    g_{\defbd}(\theta) = \sin(\theta), & & \theta = \arctan\left(\frac{x_2 - \frac{1}{2}}{x_1}\right), & & \bx = [x_1, x_2]^\top \in \defbd,
\end{align*}
while homogeneous Dirichlet boundary conditions are applied to the outer boundary $\rmbd$. Note that this choice of Dirichlet boundary condition $g_{\defbd}$ on the feature boundary leads to a singular perturbation of the PDE problem; indeed, for $|\defbd| \to 0$, the exact problem does not converge to the defeatured one. Hence, it is natural to expect that the defeaturing error does not vanish in the limit $|\defbd| \to 0$.

The results in \cref{fig:num:dirdir:shape} suggest that the estimator is reliable and the hidden constant is independent of the feature size, at least for small enough features. We remark that the shape dependence is weak and that the hidden constant is close to one.
Moreover, we indeed observe an almost constant defeaturing error with respect to the feature size, which is due to the singular perturbation introduced by the feature boundary condition $g_\defbd$. Nevertheless, the estimator accurately captures this relationship of the geometry and the imposed boundary condition.

\begin{figure}
    \centering
    \begin{tikzpicture}
    \def\side{2.5} 
    \def\spacing{3} 
    \def\DiskRadius{0.4}
    \def\SquareS{0.6}
    \def\Angle{20}
    \def\TriangleBase{0.25}
    \def\NmnMarginFactor{0.25}

    \foreach \i in {0,1,2} {
        \draw[thick, gray] (\i*\spacing,0) rectangle ++(\side,\side);
        \node at (\i*\spacing + 0.8*\side, 0.8*\side) {$\domain$};
    }

    \draw[thick, blue1] (0.5*\side-\DiskRadius,\side) arc (180:360:\DiskRadius);
    \node[blue1] at (0.5*\side, 0.75 *\side) {$\defbd$};
    \node[red1] at (0.5*\side, 1.1 *\side) {$\defbd_0$};
    \draw[thick, red1] (0.5*\side-\DiskRadius,\side) -- (0.5*\side+\DiskRadius,\side);


    \draw[thick, blue1] (\spacing + 0.5*\side - \SquareS / 2, \side) -- (\spacing + 0.5*\side - \SquareS / 2, \side - \SquareS) -- (\spacing + 0.5*\side + \SquareS / 2, \side - \SquareS) -- (\spacing + 0.5*\side + \SquareS / 2, \side);
    \draw[thick, red1] (\spacing + 0.5*\side - \SquareS / 2, \side) -- (\spacing + 0.5*\side + \SquareS / 2, \side);
    \node[blue1] at (\spacing + 0.5*\side, 0.65 *\side) {$\defbd$};
    \node[red1] at (\spacing + 0.5*\side, 1.1 *\side) {$\defbd_0$};

    \draw[gray, dotted] (2 * \spacing + 0.5*\side, \side) -- ($(2 * \spacing + 0.5 * \side, \side) - (0, \TriangleBase / tan(\Angle)$);
    \draw[gray] ($(2 * \spacing + 0.5 * \side, \side) - (0, 0.3 * \TriangleBase / tan(\Angle)$) arc (90:90-1.3 * \Angle:0.4);
    \draw[thick, blue1] (2 * \spacing + 0.5*\side - \TriangleBase, \side) -- ($(2 * \spacing + 0.5 * \side, \side) - (0, \TriangleBase / tan(\Angle)$) -- (2 * \spacing + 0.5*\side + \TriangleBase, \side);

    \node[node font=\tiny] at (2 * \spacing + 0.5*\side + 0.45 * \TriangleBase, \side - 0.1) {$\alpha$};
    
    \draw[thick, red1] (2 *\spacing + 0.5*\side - \TriangleBase, \side) -- (2 * \spacing + 0.5*\side + \TriangleBase, \side);
    
    \node[blue1] at (2*\spacing + 0.5*\side - 0.25, 0.75 *\side) {$\defbd$};
    \node[red1] at (2*\spacing + 0.5*\side, 1.1 *\side) {$\defbd_0$};  
\end{tikzpicture}
    \caption{Illustration of the boundary features considered in the experiment (L-R): Disk, square, triangle. The defeatured and simplified boundaries $\defbd$ and $\simpbd$ are indicated in blue and red, respectively.}
    \label{fig:boundary features}
\end{figure}
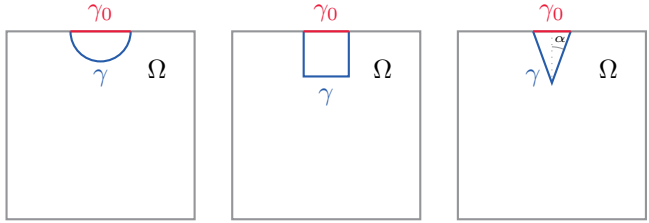

Next, we consider the dependence of the defeaturing error and the estimator on the opening angle $\alpha$ for the triangle feature; c.f. \cref{fig:boundary features}. Note that the geometry corresponds to a crack in the limit $\alpha \to 0$, and we expect a singular solution. We keep the feature size $|\defbd|$ constant while we vary the angle $\alpha$ between $1$ and $50\degree$. For the boundary data, we consider a smooth configuration with a global Dirichlet boundary condition,
\begin{align*}
    g_{\defbd}(\bx) = g_{\rmbd}(\bx) = -\frac{1}{2}(x_1^2 + x_2^2), && \bx = [x_1, x_2]^\top \in \defbd \cup \rmbd,
\end{align*}
to ensure that we do not introduce singularities at the connection between the outer and the feature boundary.
The source term is given by a sine wave in both coordinate directions, i.e.
\begin{align}
\label{eq:num:oscillatory source 2d}
    f(\bx) = A \sin(2 \pi k_1 x_1) \sin(2 \pi k_2 x_2), && \bx = [x_1, x_2]^\top \in \domain,
\end{align}
where $A = 12, k_1 = 2$, and $k_2 = 1$. The source term is extended to $\Tilde{f}$ by the same expression.

\cref{fig:num:dirdir:sing} demonstrates that the estimator remains reliable even for the smallest values of $\alpha$. In particular, the angle only weakly impacts the effectivity index, showcasing the estimator's robustness.

\begin{figure}
    \centering
        \begin{subfigure}[T]{0.45\textwidth}
        \includegraphics{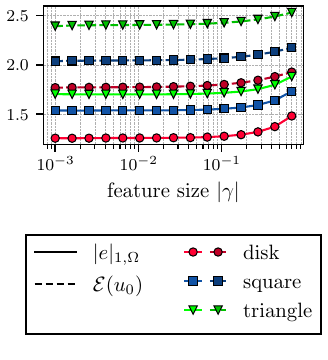}
    \caption{}
    \end{subfigure}
    \hfill
    \begin{subfigure}[T]{0.45\textwidth}
        \includegraphics{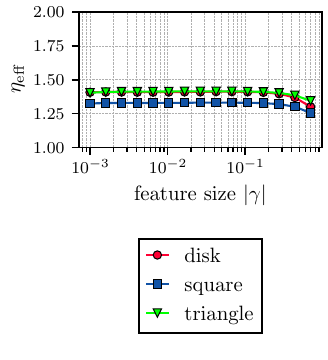}
    \caption{}
    \end{subfigure}

    \caption{Defeaturing errors and estimates (a) and estimator effectivities (b) as a function of the feature size for the feature shapes illustrated in \cref{fig:boundary features} for the Dirichlet-Dirichlet case in two-dimensions.}
    \label{fig:num:dirdir:shape}
\end{figure}

\begin{figure}
    \centering
        \begin{subfigure}[T]{0.45\textwidth}
        \includegraphics{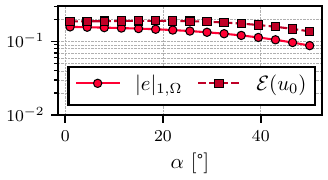}
    \caption{}
    \end{subfigure}
    \hfill
    \begin{subfigure}[T]{0.45\textwidth}
        \includegraphics{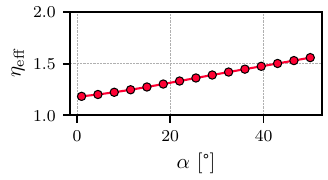}
    \caption{}
    \end{subfigure}

    \caption{Defeaturing errors and estimates (a) and estimator effectivities (b) as a function of the triangle feature's opening angle $\alpha$ shown in \cref{fig:internal features}.}
    \label{fig:num:dirdir:sing}
\end{figure}

\paragraph{Three-dimensional geometries}
 Consider the following setup: The defeatured domain is the unit cube $\defdomain = [-\tfrac 1 2, \tfrac 1 2]^3 \subset \R^3$ with negative features cut out from the top side, as shown in \cref{fig:boundary features 3d}. We impose a global Dirichlet boundary condition on all boundaries:
\begin{align*}
    g_{\defbd}(\bx) = g_{D}(\bx) = g_{\simpbd}(\bx) = -\frac{1}{2}(x_1^2 + x_2^2 + x_3^2), && \bx = [x_1, x_2, x_3]^\top \in \partial \domain.
\end{align*}
The source term is set to
\begin{align}
\label{eq:num:oscillatory source 3d}
f(\bx) = A \sin(2 \pi k_1 x_1) \sin(2 \pi k_2 x_2) \sin(2 \pi k_3 x_3), && \bx = [x_1, x_2, x_3]^\top \in \domain,
\end{align}
where $A = 1, k_1 = 1$, $k_2 = 2$, and $k_3 = 3$. The source term is extended to $\Tilde{f}$ by the same expression.

\Cref{fig:num:dirdir:shape3d} shows that the estimator is reliable and that the hidden constant is independent of the feature's size for small enough features. Moreover, the hidden constant depends on the feature shape only weakly, as in the two-dimensional experiments. In contrast, the defeaturing error vanishes in this experiment for $|\defbd| \to 0$ since choosing a globally continuous Dirichlet boundary condition implies that the feature merely introduces a regular perturbation of the PDE problem.

\begin{figure}
    \centering
        \begin{subfigure}[T]{0.45\textwidth}
        \includegraphics[scale=0.042]{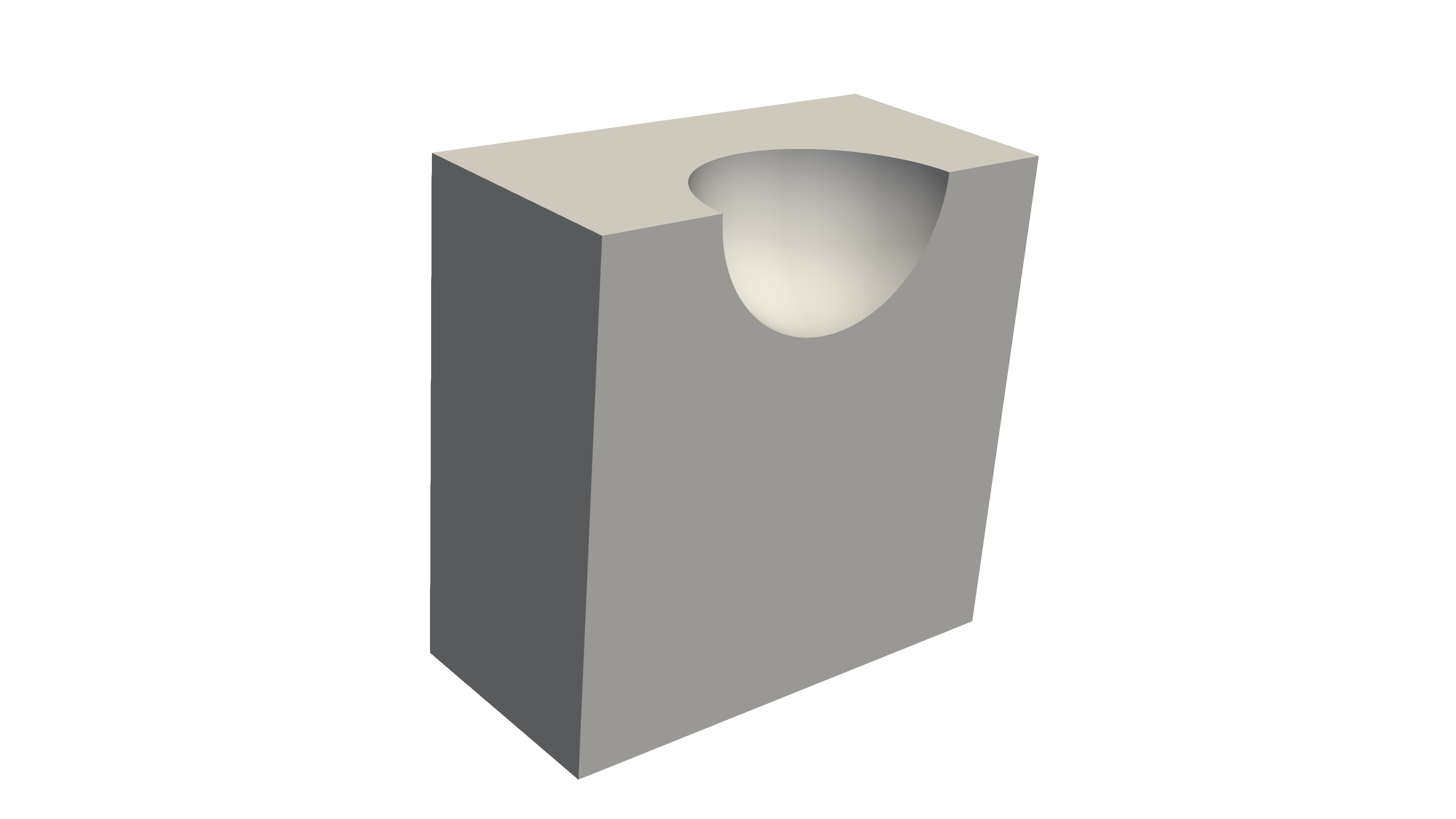}
    \caption{}
    \end{subfigure}
    \hfill
    \begin{subfigure}[T]{0.45\textwidth}
        \includegraphics[scale=0.042]{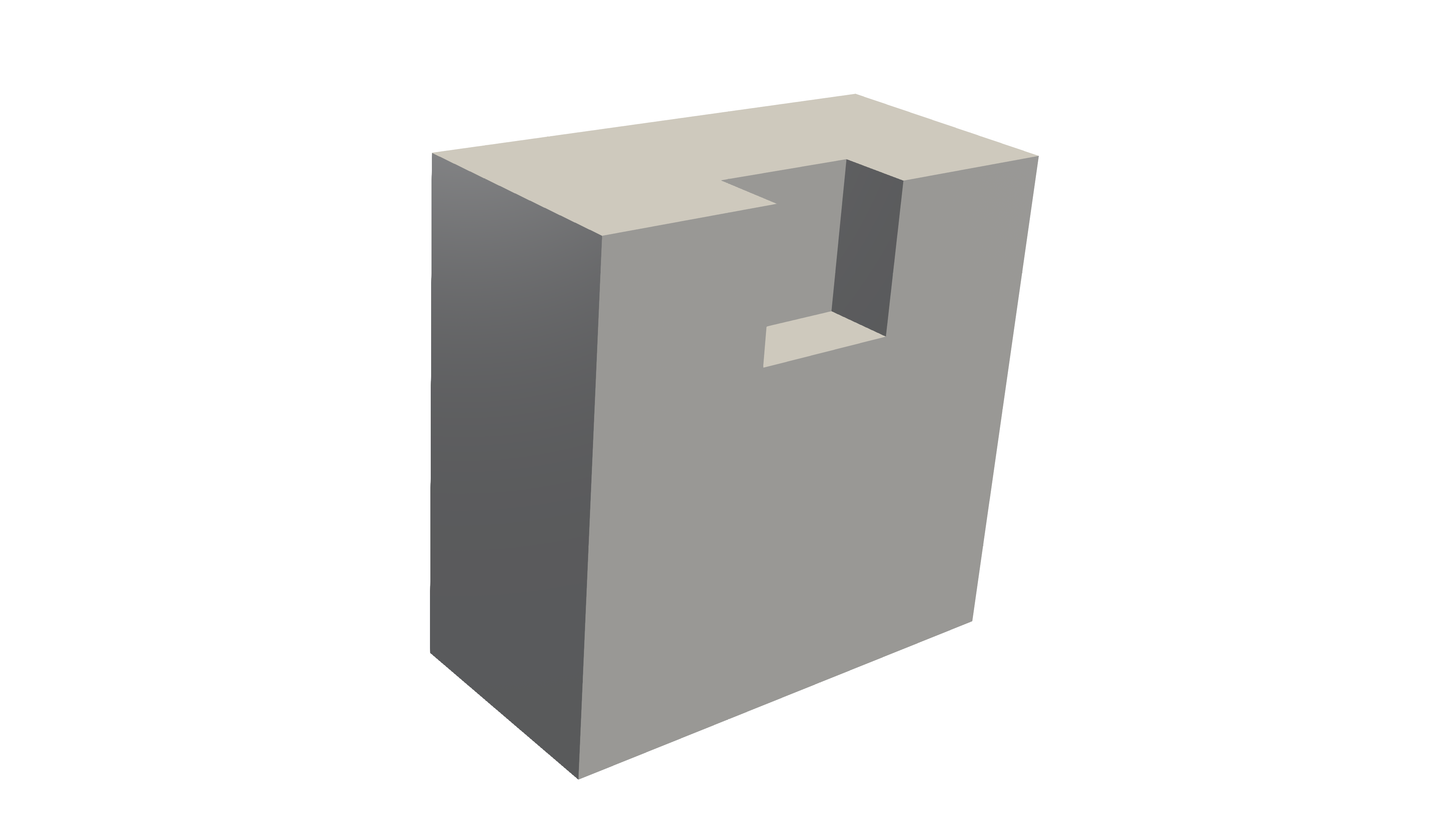}
    \caption{}
    \end{subfigure}

    \caption{Cut in the $y-z$ plane of the geometries with sphere (a) and cube boundary feature (b) considered in the experiments in three dimensions.}
    \label{fig:boundary features 3d}
\end{figure}

\begin{figure}
    \centering
        \begin{subfigure}[T]{0.45\textwidth}
        \includegraphics{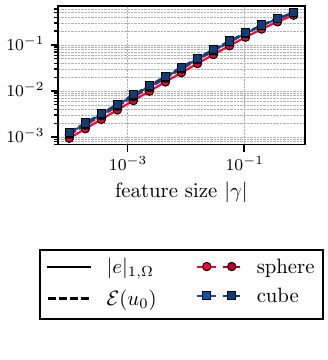}
    \caption{}
    \end{subfigure}
    \hfill
    \begin{subfigure}[T]{0.45\textwidth}
        \includegraphics{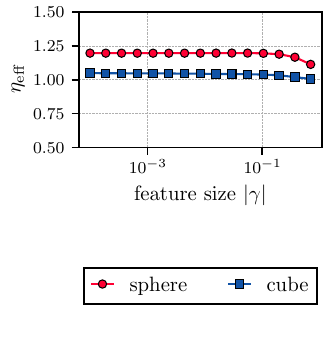}
    \caption{}
    \label{fig:num:dirdir:shape3d:eff}
    \end{subfigure}

    \caption{Defeaturing errors and estimates (a) and estimator effectivities (b) as a function of the feature size for the feature shapes illustrated in \cref{fig:boundary features 3d} for the Dirichlet-Dirichlet case in three dimensions.}
    \label{fig:num:dirdir:shape3d}
\end{figure}

\subsection{Dirichlet-Neumann features}
\label{subsec:num:dirichlet neumann features}
As for the case of Dirichlet-Dirichlet features, we validate the estimator proposed in \cref{eq:dirichlet neumann estimator} on different feature shapes and sizes.

\paragraph{Two-dimensional geometries}
We test the size and shape dependency of the defeaturing error and estimator in a similar way as in the Dirichlet-Dirichlet case: We consider again the boundary features illustrated in \cref{fig:boundary features}, with the opening angle $\alpha = 15\degree$ for the triangle. However, this time, we apply Neumann boundary conditions to the top side of the square denoted by $\rmbd_N$, which touches the feature. The remaining three sides form the outer Dirichlet boundary $\rmbd_D$. As problem data, we prescribe a constant source term $f \equiv 1$ in $\domain$ and homogeneous Dirichlet and Neumann boundary conditions on $\rmbd_D$ and $\rmbd_N$, respectively.
On the feature boundary $\defbd$, we prescribe the Dirichlet boundary condition
\begin{align*}
    g_{\defbd}(\theta) = \cos(\theta), & & \theta = \arctan\left(\frac{x_2 - \frac{1}{2}}{x_1}\right), & & \bx = [x_1, x_2]^\top \in \defbd.
\end{align*}
Finally, when solving the defeatured problem, we extend the homogeneous Neumann boundary conditions to the defeatured boundary $\simpbd$. This is a natural choice,
since otherwise in practical applications one would have to precisely mesh the intersection of the outer and feature boundaries, which may be as costly as meshing the feature in the first place.

The results are visualized in \cref{fig:num:dirneum:shape}. The estimator is reliable for all sizes and shapes. The hidden constant is independent of the feature size, at least for small enough features. Furthermore, there is only a weak dependence on feature's shape. Finally, we point out that the defeaturing error remains approximately constant independent of the size, similar to the Dirichlet-Dirichlet case.

\begin{figure}
    \centering
        \begin{subfigure}[T]{0.45\textwidth}
        \includegraphics{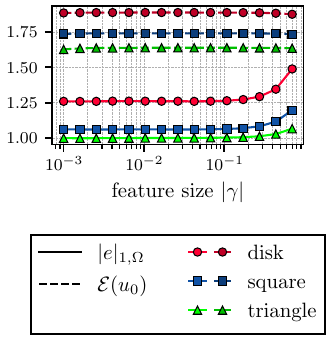}
    \caption{}
    \end{subfigure}
    \hfill
    \begin{subfigure}[T]{0.45\textwidth}
        \includegraphics{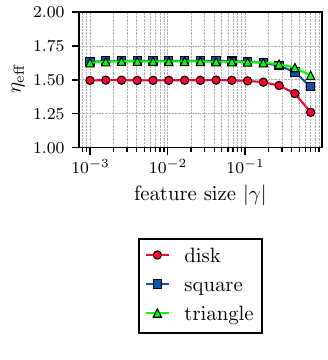}
    \caption{}
    \end{subfigure}

    \caption{Defeaturing errors and estimates (a) and estimator effectivities (b) as a function of the feature size for the feature shapes illustrated in \cref{fig:boundary features} for the Dirichlet-Neumann case.}
    \label{fig:num:dirneum:shape}
\end{figure}

\paragraph{Three-dimensional geometries}
We again consider the three-dimensional geometries in \cref{fig:boundary features 3d}. The Neumann boundary $\rmbd_N$ is the part of $\rmbd$ at a distance smaller than 0.3 from the intersection line $\overline{\defbd} \cap \overline{\rmbd}$. We impose the source term \cref{eq:num:oscillatory source 3d} from the Dirichlet-Dirichlet problem. In addition, we apply homogeneous Dirichlet and Neumann boundary conditions to $\rmbd_D, \rmbd_N$, and $\simpbd$, respectively, as for the two-dimensional example. The Dirichlet boundary conditions on $\defbd$ are set to
\begin{align*}
    g_\defbd(\bx) = x_1^2 + x_3^2, && \bx = [x_1, x_2]^\top \in \defbd.
\end{align*}
The results for different feature sizes are plotted in \cref{fig:num:dirneum:shape3d}. We observe that the estimator is reliable and that the hidden constant is independent of the feature size for small enough features. However, comparing the effectivities in \cref{fig:num:dirneum:shape3d:eff} to the ones in the Dirichlet-Dirichlet case in \cref{fig:num:dirdir:shape3d:eff}, we note that the initial range of feature sizes, where there is some higher-order shape dependency, is longer in the Dirichlet-Neumann case. The latter is due to more higher-order terms being ignored in the derivation of the estimator \cref{eq:dirichlet neumann estimator} due to the extension of $\bderr$ to $\defbd \cup \rmbd_N$; see \cref{thm:dirichlet neumann reliability,subsec:extensions}.

\begin{figure}
    \centering
        \begin{subfigure}[T]{0.45\textwidth}
        \includegraphics{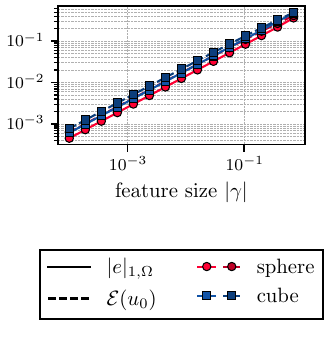}
    \caption{}
    \end{subfigure}
    \hfill
    \begin{subfigure}[T]{0.45\textwidth}
        \includegraphics{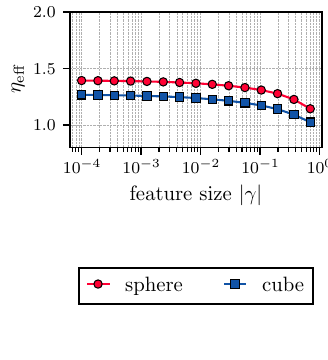}
    \caption{}
    \label{fig:num:dirneum:shape3d:eff}
    \end{subfigure}

    \caption{Defeaturing errors and estimates (a) and estimator effectivities (b) as a function of the feature size for the feature shapes illustrated in \cref{fig:boundary features 3d} for the Dirichlet-Neumann case.}
    \label{fig:num:dirneum:shape3d}
\end{figure}

\subsection{Internal features}
Let us finally investigate the performance of the estimator for internal features defined in \cref{eq:internal estimator}. As for the other feature types, we investigate the shape dependency of the hidden constant. Furthermore, we illustrate that the defeaturing estimator is agnostic to the choice of extension for the source term, while the latter substantially influences the size of the defeaturing error as a whole.

\subsubsection{Shape dependency}

\paragraph{Two-dimensional geometries}
We start with two-dimensional geometries consisting of a unit square with a negative feature at the center. \Cref{fig:internal features} illustrates the different feature types.

We apply the defeaturing estimator to the following setup: On the outer boundary $\rmbd$, consider homogeneous Dirichlet boundary conditions, i.e., $g_{D} \equiv 0$. On the feature boundary $\defbd$, we impose the Dirichlet data
\begin{align*}
    g_{\defbd}(\theta) = \cos(\theta) + 1, && \theta = \arctan(x_2/x_1), && \bx = [x_1, x_2]^\top \in \defbd.
\end{align*}
We point out that $g_{\defbd}$ is not the trace of a continuous function over $\domain$. Hence, the internal feature introduces a singular perturbation to the PDE problem, and therefore, we cannot expect the defeaturing error to vanish in the limit $|\defbd| \to 0$. Finally, we impose the constant source term $f \equiv 1$ over $\domain$, which we naturally extend to $\Tilde{f} \equiv 1$ on $\defdomain$.

The defeaturing errors for different feature sizes are plotted next to their corresponding defeaturing error estimates in \cref{fig:num:internal:shape:abs}. The defeaturing errors remain approximately constant for all feature types, as expected. Moreover, the estimator proves to be reliable in all cases. The estimator effectivity $\esteff$ is plotted in \cref{fig:num:internal:shape:eff} and is approximately constant and close to one for all feature types. We point out, however, that the chosen Dirichlet data on the features generate very singular solutions. These are difficult to resolve for the less regular features, square, star, and L-shape, and require extremely fine meshes manifesting themselves in the slightly slanted effectivity lines.

Furthermore, the effectivities for different feature types are very close, suggesting that the shape dependence is weak. Nevertheless, we note that all features are normalized to the same circumference. However, C-shape and L-shape result in lower errors and sharper estimates. This suggests that there might also be an important dependency on the feature's mass, which is currently not tracked by the estimator.

\begin{figure}
    \centering
    \begin{tikzpicture}
    \def\side{2.25} 
    \def\spacing{2.6} 
    \def\pival{3.14159}
    \def\nvertices{5} 
    \def\angle{360/\nvertices} 
    \def\outerr{0.5} 
    \def\innerr{0.25} 
    \def\CouterR{0.4} 
    \def\CinnerR{0.05} 
    \def\Lunit{0.07}

    \foreach \i in {0,1,2,3,4} {
        \draw[thick, gray] (\i*\spacing,0) rectangle ++(\side,\side);
        \node at (\i*\spacing + 0.8*\side, 0.8*\side) {$\domain$};
        \node[blue1] at (\i*\spacing + 0.5*\side, 0.17*\side) {$\defbd$};
    }
    
    \draw[thick, blue1] (0.5*\side,0.5*\side) circle (\side / 6);
    
    \draw[thick, blue1] (\spacing*1+0.5*\side - 0.25,0.5*\side - 0.25) rectangle (\spacing*1+0.5*\side + 0.25,0.5*\side + 0.25);
    
    \begin{scope}
        \foreach \i in {0,...,4} {
            \coordinate (S\i) at ($(\spacing*2+0.5*\side,0.5*\side) + (\i*\angle:\outerr)$);
            \coordinate (I\i) at ($(\spacing*2+0.5*\side,0.5*\side) + (\i*\angle + \angle/2:\innerr)$);
        }
        \draw[thick, blue1] (S0) -- (I0) \foreach \i in {1,...,4} { -- (S\i) -- (I\i) } -- cycle;
    \end{scope}
    
    \begin{scope}
        \draw[thick, blue1] (\spacing*3+0.5*\side,0.5*\side + \CouterR + 2 * \CinnerR) arc (90:270:\CouterR + 2 * \CinnerR);
        \draw[thick, blue1] (\spacing*3+0.5*\side,0.5*\side+\CouterR) arc (90:270:\CouterR);
        \draw[thick, blue1] (\spacing*3+0.5*\side,0.5*\side-\CouterR-2*\CinnerR) arc (90:270:-\CinnerR);
        \draw[thick, blue1] (\spacing*3+0.5*\side,0.5*\side+\CouterR) arc (90:270:-\CinnerR);
    \end{scope}
    
    \coordinate (L1) at ($(\spacing*4+0.5*\side,0.5*\side) + (4 * \Lunit, - 5 * \Lunit)$);
    \coordinate (L2) at ($(\spacing*4+0.5*\side,0.5*\side) + (4 * \Lunit, - 3 * \Lunit)$);
    \coordinate (L3) at ($(\spacing*4+0.5*\side,0.5*\side) + ( -\Lunit, - 3 * \Lunit)$);
    \coordinate (L4) at ($(\spacing*4+0.5*\side,0.5*\side) + ( -\Lunit,  5 * \Lunit)$);
    \coordinate (L5) at ($(\spacing*4+0.5*\side,0.5*\side) + ( -3 * \Lunit,  5 * \Lunit)$);
    \coordinate (L6) at ($(\spacing*4+0.5*\side,0.5*\side) + ( -3 * \Lunit,  -5 * \Lunit)$);

    \draw[thick, blue1] (L1) -- (L2) -- (L3) -- (L4) -- (L5) -- (L6) -- (L1);
\end{tikzpicture}
    \caption{Comparison of the internal feature shapes considered in the experiment (L-R): Disk, square, star, C-shape, L-shape. The defeatured boundary $\defbd$ is indicated in blue.}
    \label{fig:internal features}
\end{figure}
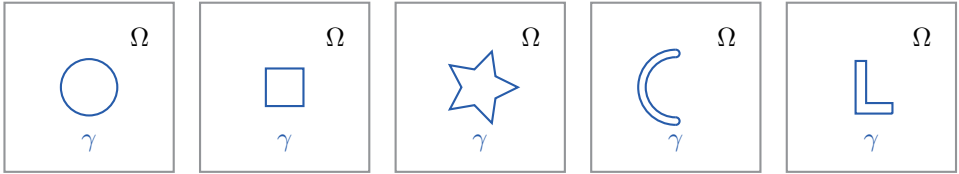

\begin{figure}
    \centering
        \begin{subfigure}[T]{0.45\textwidth}
        \includegraphics{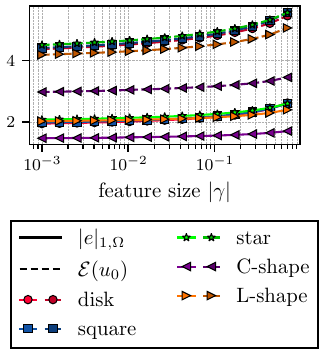}
    \caption{}
    \label{fig:num:internal:shape:abs}
    \end{subfigure}
    \hfill
    \begin{subfigure}[T]{0.45\textwidth}
        \includegraphics{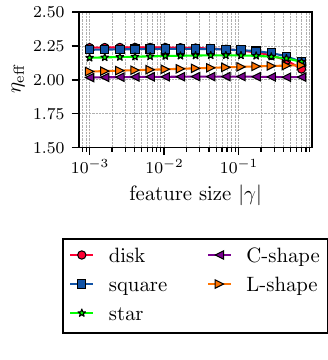}
    \caption{}
    \label{fig:num:internal:shape:eff}
    \end{subfigure}

    \caption{Defeaturing errors and estimates (a) and estimator effectivities (b) as a function of the feature size for the feature shapes illustrated in \cref{fig:internal features}.}
    \label{fig:num:internal:shape}
\end{figure}

\paragraph{Three-dimensional geometries}

\begin{figure}
    \centering
        \begin{subfigure}[T]{0.45\textwidth}
        \includegraphics[scale=0.042]{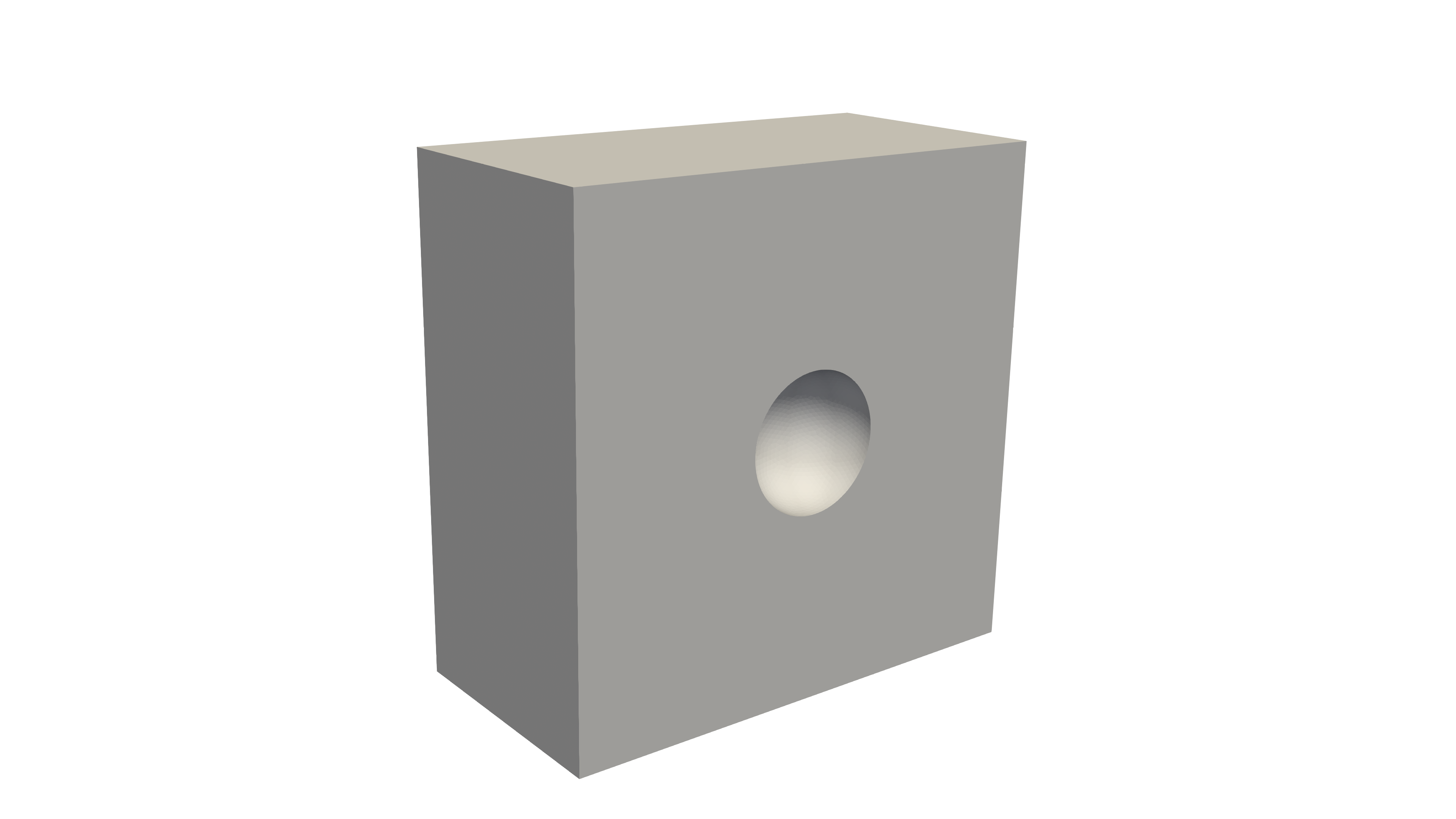}
    \caption{}
    \end{subfigure}
    \hfill
    \begin{subfigure}[T]{0.45\textwidth}
        \includegraphics[scale=0.042]{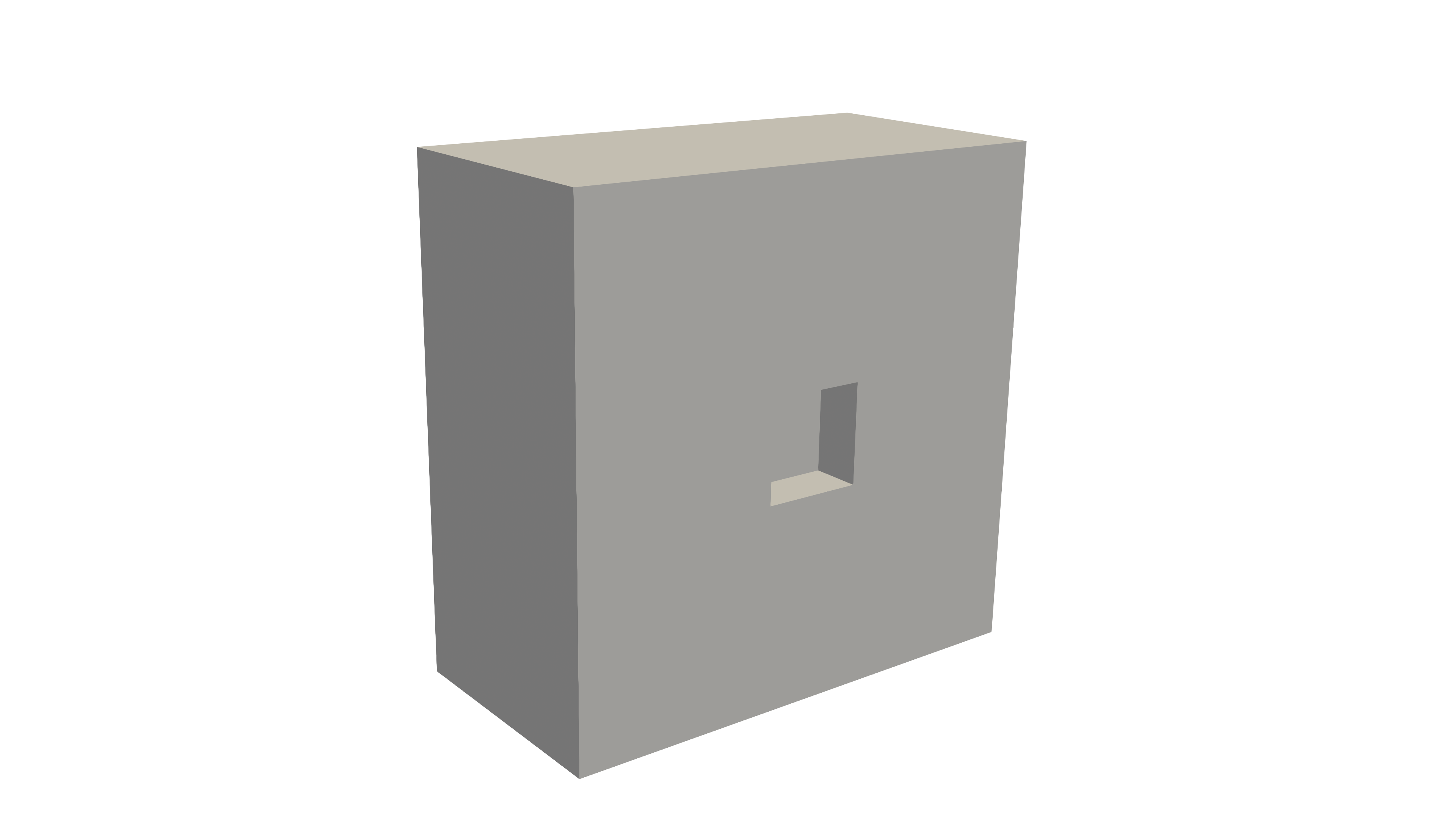}
    \caption{}
    \end{subfigure}

    \caption{Cut in the $y-z$ plane of the geometries with internal sphere (a) and cube feature (b) considered in the experiments in three dimensions.}
    \label{fig:internal features 3d}
\end{figure}

Consider the geometries displayed in \cref{fig:internal features 3d}. We impose once more the source term \cref{eq:num:oscillatory source 3d} and homogeneous Dirichlet boundary conditions on the outer boundary $\rmbd$. On the feature, we apply the boundary condition
\begin{align}
    g_\defbd(\bx) = \exp(\cos(x_1) \sin(x_2)), & & \bx = [x_1, x_2]^\top \in \defbd.
\end{align}
The defeaturing errors and effectivity indexes are shown in \cref{fig:num:internal:shape3d}. The estimator is reliable for both feature types across all sizes. The effectivity index converges to a constant value larger than one for both feature types. For the spherical feature, the effectivity index is only slightly larger than one, because the contribution by $\defestintavg{u_0}$ dominates. In fact, it can be shown that the average component $\defestintavg{u_0}$ of the internal estimator is sharp for spherical shells. Hence, this result is expected. 

\begin{figure}
    \centering
        \begin{subfigure}[T]{0.45\textwidth}
        \includegraphics{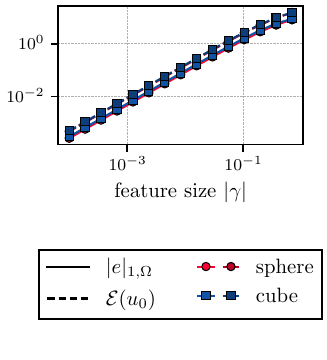}
    \caption{}
    \label{fig:num:internal:shape3d:abs}
    \end{subfigure}
    \hfill
    \begin{subfigure}[T]{0.45\textwidth}
        \includegraphics{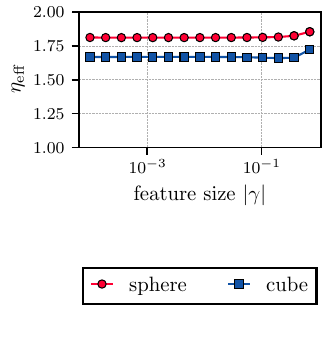}
    \caption{}
    \label{fig:num:internal:shape3d:eff}
    \end{subfigure}

    \caption{Defeaturing errors and estimates (a) and estimator effectivities (b) as a function of the feature size for the internal feature case in three dimensions}
    \label{fig:num:internal:shape3d}
\end{figure}

\subsubsection{Source term extensions}
In contrast to the boundary feature cases, it is well-known from singular perturbation theory that naively extending the source term into the feature by the same expression does not yield a good approximation of the exact solution regardless of the boundary condition \cite{mazya_asymptotic_2000}. Indeed, in the limit $|\genbd| \to 0$, the exact PDE \cref{eq:exact poisson problem} does not converge to its defeatured variant \cref{eq:internal defeatured problem}, but rather to a PDE on a perforated domain.

In the proposed defeaturing estimator this phenomenon is reflected by the coefficient $\Bar{c}_\defbd$, which vanishes only very slowly for $|\defbd| \to 0$. Indeed, even in a domain with a smooth boundary $\partial \domain$ and smooth boundary conditions, the average $\avg{\bderr}{\defbd}$ of the boundary error may still be non-negligible. Hence, $\defestintavg{u_0} = \Bar{c}_\defbd |\avg{\bderr}{\defbd}|$ only converges to zero very slowly, while the non-average part $\defestintnavg{u_0}$ vanishes quickly.

However, by adding an appropriate correction term to the right-hand side of the defeatured problem, one obtains a much better first-order approximation.
Let $\hat{G}(x) = \frac{1}{2\pi}\log(|x - m_\feat|)$ denote the fundamental solution of the Laplacian centered at $m_\feat$ and consider $G(x) = \hat{G}(x) - g(x)$, where $g$ is the solution to
\begin{align}
\label{eq:num:green}
    \begin{cases}
        -\Delta g = 0,  \text{ in } \defdomain,\\
        g(x) = -\hat{G}(x)  \text{ on } \rmbd_D, \\
        \partial_n g(x) = -\partial_n\hat{G}(x)  \text{ on } \rmbd_N. \\
    \end{cases}
\end{align}
If the feature $\feat$ is a disk centered at $m_\feat$, the first-order approximation is given by
\begin{align}
\label{eq:internal first order approximation}
    u_1(x) \coloneqq u_0(x) + \mu(\defbd) \avg{\bderr}{\defbd} G(x), & & \mu(\defbd) \coloneqq \frac{2\pi}{\log(\diam\defbd/2) - 2\pi \avg{g}{\defbd}},
\end{align}
while for more general shapes, $\mu(\defbd)$ additionally depends on the logarithmic capacity of the feature \cite{mazya_asymptotic_2000, ransford_computation_2011}.

In the context of the estimator, note that $-\Delta u_1 = f$ still holds in $\domain$, while there is a point source at the barycenter $m_\feat$ of the feature. Consequently, the reliability in \cref{thm:internal estimate reliability} still holds.
Furthermore, we have in the case of a disk feature that
\begin{align*}
    \avg{(u - u_1)}{\defbd} = (1 - \mu(\defbd) \avg{G}{\defbd}) \avg{\bderr}{\defbd} = 0,
\end{align*}
while for more general shapes we have $\mu(\defbd) \avg{G}{\defbd} \to 1$, for $|\defbd| \to 0$. Therefore, the first-order approximation corresponds to adding a point source at the barycenter to remove the average of the boundary error.

Let us illustrate this with the following experiment: Consider again the square domain with a disk cut from the center in \cref{fig:internal features}. We impose the Dirichlet data $g_D \equiv 1$ and $g_\defbd(\bx) = x_1^2 + x_2$ on $\rmbd$ and $\defbd$, respectively. The source term is given by \cref{eq:num:oscillatory source 2d}. Now we compute the naively defeatured solution $u_0$ and the corrected solution $u_1$, where for the latter we must additionally solve equation \cref{eq:num:green}.

\begin{figure}
    \centering
        \begin{subfigure}[T]{0.45\textwidth}
        \includegraphics{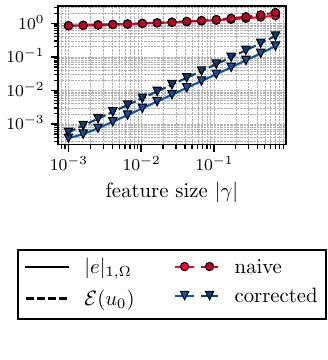}
    \caption{}
    \label{fig:num:internal:pert:abs}
    \end{subfigure}
    \hfill
    \begin{subfigure}[T]{0.45\textwidth}
        \includegraphics{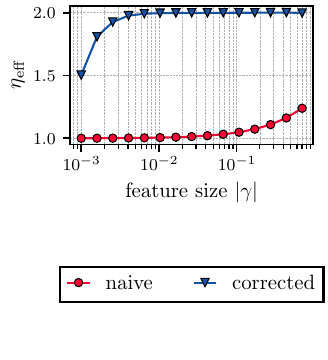}
    \caption{}
    \label{fig:num:internal:pert:eff}
    \end{subfigure}

    \caption{Defeaturing errors and estimates (a) and estimator effectivities (b) as a function of the feature size for naive and corrected source term extensions. The drop of $\eta_{\mathrm{eff}}$ at the lower end for the corrected solution is due to numerical interpolation issues of the fundamental solution rather than the properties of the estimator.}
    \label{fig:num:internal:pert}
\end{figure}

The results are shown in \cref{fig:num:internal:pert}. We see that the estimator remains reliable for the corrected solution, while the corresponding defeaturing error vanishes quickly, in contrast to the naively defeatured solution. This shows that the estimator is indeed agnostic to the extension of the source term, provided that the defeatured solution satisfies the original PDE in the exact domain. However, from a computational point of view, additionally solving for $g$ for each feature is clearly not efficient. Furthermore, interpolating the fundamental solution $\hat{G}$ requires a heavily refined mesh around small features, which might be as costly as meshing the feature in the first place. The latter is reflected by the drop of the effectivity index for the corrected solution in \cref{fig:num:internal:pert:eff}.

\section{Conclusion}
\label{sec:conclusion and outlook}
In this work, we have extended the framework for a posteriori defeaturing error estimation in \cite{buffa_analysis-aware_2022} to negative Dirichlet features in the interior and on the boundary. The two crucial elements were the scaling behavior of the Neumann trace operator norm with respect to the defeatured boundary's size and the continuous extension on Lipschitz boundaries. The estimators were successfully validated numerically in all three situations of \cref{problem:defeatured problem} in two and three dimensions, demonstrating that they accurately capture the error's dependence on the feature geometry and the underlying differential problem. Notably, the estimators are explicit in the feature size, independent of the specific extension chosen for the source term, and computationally inexpensive as they only require boundary integrals over the defeatured boundary.

While this paper paper focuses on the analysis in continuous spaces for a single feature, the framework is readily extensible.
Since the estimator is merely additive, extending the estimators to multiple features is straightforward, as in the Neumann case \cite{antolin_analysisaware_2024}. The discretization error can be taken into account via an a posteriori error estimator simultaneously; see \cite{chanon_adaptive_2022} and \cite{buffa_equilibrated_2024}. Positive features can be handled by solving an extended problem in the bounding box of the positive feature as proposed in \cite{buffa_analysis-aware_2022, chanon_adaptive_2022}. Finally, we only consider the global energy norm of the error in this paper. Deriving goal-oriented defeaturing error estimates will be an important step in applying the method in industrial settings. The development of such goal-oriented error estimates is the subject our future work.

\appendix
\section{Some results on Sobolev trace spaces}
\label{sec:results on trace spaces}
In this section, we state several technical results that are used in the proofs of \cref{sec:estimates}.

First, we recall the following scaling properties of the fractional Sobolev norm defined in \cref{eq:definition of fractional sobolev norm}. Let $\genbd$ be an $(n-1)$-dimensional Lipschitz submanifold of $\R^n$ and $\mu \in \hs{\genbd}$ with $s \in (0, 1)$. For $\lambda > 0$ we denote by $\genbd_\lambda$ the submanifold obtained from rescaling $\R^n$ by the factor $\lambda$. Similarly, we write $x_\lambda = \lambda x$ and denote by $\mu_\lambda$ the push-forward of $\mu$ defined by 
\begin{align*}
  \mu_\lambda(x_\lambda) = \mu_\lambda(\lambda x) \coloneqq \mu(x), \quad x \in \genbd.
\end{align*}
Then, using the usual change of variable formula, we find
\begin{align}
\label{eq:ltwo norm scaling}
    \Ltwonorm{\mu_\lambda}{\genbd_\lambda}^2 = \int_{\genbd_\lambda} |\mu_\lambda(x_\lambda)|^2 \dd{s(x_\lambda)} = \lambda^{n-1} \int_\genbd |\mu(x)|^2 \dd{s(x)} = \lambda^{n-1}\Ltwonorm{\mu}{\genbd}^2.
\end{align}
Similarly, we find
\begin{nalign}
\label{eq:hs seminorm scaling}
    \seminorm{\mu_\lambda}{s}{\genbd_\lambda}^2 &= \int_{\genbd_\lambda} \int_{\genbd_\lambda} \frac{|\mu_\lambda(x_\lambda) - \mu_\lambda(y_\lambda)|^2}{|x_\lambda - y_\lambda|^{n - 1 + 2s}} \dd{s(x_\lambda)}\dd{s(y_\lambda)}\\
    &= \frac{\lambda^{2(n - 1)}}{\lambda^{n - 1 + 2s}} \int_\genbd \int_\genbd \frac{|\mu(x) - \mu(y)|^2}{|x - y|^{n - 1 + 2s}} \dd{s(x)}\dd{s(y)} = \lambda^{n - 1 - 2s} \seminorm{\mu}{s}{\genbd}^2.
\end{nalign}
These observations allow us to formulate the following statement:

\begin{lemma}
\label{lemma:scaling of sobolev norms}
    Let $\genbd$ be an $(n-1)$-dimensional Lipschitz submanifold of $\R^n$ and $\mu \in \hs{\genbd}$ with $s \in (0, 1)$. Denote by $\hat{\genbd}$ the same submanifold rescaled to unit size such that $|\hat{\genbd}| = 1$ and by $\hat{\mu}$ the push-forward of $\mu$ to $\hs{\hat{\genbd}}$. Then, we have
    \begin{align*}
        \Ltwonorm{\mu}{\genbd} = |\genbd|^{-\frac{1}{2}} \Ltwonorm{\hat{\mu}}{\hat{\genbd}} && \text{ and } && \seminorm{\mu}{s}{\genbd} = |\genbd|^{\frac{1 + 2s - n}{2(n-1)}} \seminorm{\hat{\mu}}{s}{\hat{\genbd}}.
    \end{align*}
\end{lemma}

\begin{proof}
    We obtain $\hat{\genbd}$ by rescaling $\R^n$ by the factor $|\genbd|^{\frac{-1}{n-1}}$. Replacing $\lambda$ by $|\genbd|^{\frac{-1}{n-1}}$ in \cref{eq:ltwo norm scaling} and \cref{eq:hs seminorm scaling} yields the desired result.
\end{proof}

\subsection{Poincaré and interpolation inequalities}
\label{subsec:poincare and interpolation}
In this section, we state useful inequalities in trace spaces for reference with explicit dependence of any constant on the size of the underlying piece of the boundary. The following Poincaré-type inequalities stem from \cite{chanon_adaptive_2022}, while the interpolation-type inequalities in \cref{eq:interpolation and poincare inequality} are partially new and require a proof.

\begin{lemma}[Lemma 2.3.6. in \cite{chanon_adaptive_2022}]
\label{lemma:poincare I}
Assume that $\genbd \subset \partial \domain$  for some Lipschitz domain $\domain \subset \R^n$. Assume that $\genbd$ is isotropic according to \cref{def:isotropic subset} and connected, and $\partial \genbd \neq \emptyset$. Then, for all $\mu \in \htracedbz{\genbd}$,
\begin{align*}
    \Ltwonorm{\mu}{\genbd} \lesssim |\genbd|^{\frac{1}{2(n-1)}} \htraceseminorm{\mu^\star}{\partial \domain} \leq |\genbd|^{\frac{1}{2(n-1)}} \htracedbznorm{\mu}{\genbd}.
\end{align*}
\end{lemma}

We underline the necessity of the assumption $\partial \genbd \neq \emptyset$. Indeed, if $\genbd$ is a closed curve, the space $\htracedbz{\genbd}$ contains non-trivial constants, and therefore, the inequality in \cref{lemma:poincare I} does not hold. Nevertheless, the following Poincaré-type inequality holds:

\begin{lemma}[Lemma 2.3.8. in \cite{chanon_adaptive_2022}]
\label{lemma:poincare II}
    Assume that $\genbd$ is isotropic according to \cref{def:isotropic subset}. Then, for all $\mu \in \htrace{\genbd}$,
    \begin{align*}
        \Ltwonorm{\mu - \avg{\mu}{\genbd}}{\genbd} \lesssim |\genbd|^{\frac{1}{2(n-1)}} \htraceseminorm{\mu}{\genbd}.
    \end{align*}
\end{lemma}

The following result allows us to express the defeaturing error estimates in terms of easily computable quantities:
\begin{lemma}
\label{eq:interpolation and poincare inequality}
    Assume that $\genbd$ is isotropic according to \cref{def:isotropic subset}. Then, for all $\mu \in \htrace{\genbd}$,
    \begin{align}
    \label{eq:interpolation I}
        \htraceseminorm{\mu - \avg{\mu}{\genbd}}{\genbd} \lesssim \sqrt{\Ltwonorm{\mu - \avg{\mu}{\genbd}}{\genbd} \Ltwonorm{\nabla_t \mu}{\genbd}},
    \end{align}
    where $\nabla_t \mu$ denotes the gradient in tangential direction. 
    Moreover, if $\mu \in \htracedbz{\genbd}$, we have
    \begin{align}
    \label{eq:interpolation II}
        \htracedbzseminorm{\mu}{\genbd} \lesssim \sqrt{\Ltwonorm{\mu}{\genbd} \Ltwonorm{\nabla_t\mu}{\genbd}}.
    \end{align}
\end{lemma}

\begin{proof}
    Denote by $\hat{\genbd}$ the same boundary $\genbd$ rescaled to unit size and consider the push-forward $\hat{\mu} \in \htrace{\hat{\genbd}}$ as in \cref{lemma:scaling of sobolev norms}.
    First, we show that \cref{eq:interpolation I} and \cref{eq:interpolation II} for the rescaled $\hat{\genbd}$. Then, we show that the constants only depend on the shape of $\genbd$, but not its size.

    For the first step, we recall the interpolation property of fractional Sobolev spaces, which also holds on Lipschitz manifolds; see \cite{triebel_theory_1992, triebel_function_2002,schneider_beyond_2021}. That is, we have
    \begin{align*}
        \htrace{\hat{\genbd}} = [\Ltwo{\hat\genbd}, \hone{\hat{\genbd}}]_{1/2},
    \end{align*}
    up to norm equivalence, where $[\cdot, \cdot]_s$ denotes the complex interpolation functor for function spaces. The interpolation property implies in particular the existence of a constant $C_i(\hat{\genbd}) > 0$, such that
    \begin{align}
    \label{eq:interpolation poincare proof - interpolation inequality}
        \htraceseminorm{\hat{\mu} - \avg{\hat{\mu}}{\hat{\genbd}}}{\hat{\genbd}} \leq C_i(\hat{\genbd}) \sqrt{\Ltwonorm{\hat{\mu} - \avg{\hat{\mu}}{\hat{\genbd}}}{\hat{\genbd}} \honenorm{\hat{\mu} - \avg{\hat{\mu}}{\hat{\genbd}}}{\hat{\genbd}}}, & & \forall \hat{\mu} \in \htrace{\hat{\genbd}}.
    \end{align}
    The Poincaré-Friedrichs inequality implies the existence of a constant $C_p(\hat{\genbd}) > 0$ such that
    \begin{align*}
        \honenorm{\hat{\mu} - \avg{\hat{\mu}}{\hat{\genbd}}}{\hat{\genbd}} \leq C_p(\hat{\genbd}) \Ltwonorm{\nabla_t\hat{\mu}}{\hat{\genbd}}.
    \end{align*}
    By reinserting this expression into \cref{eq:interpolation poincare proof - interpolation inequality} and simplifying terms, we find a constant $C(\hat{\genbd}) > 0$ such that
    \begin{align}
    \label{eq:interpolation poincare proof - rescaled inequality I}
        \htraceseminorm{\hat{\mu} - \avg{\hat{\mu}}{\hat{\genbd}}}{\hat{\genbd}} \leq C(\hat{\genbd}) \sqrt{\Ltwonorm{\hat{\mu} - \avg{\hat{\mu}}{\hat{\genbd}}}{\hat{\genbd}} \Ltwonorm{\nabla_t \hat{\mu}}{\hat{\genbd}}}, & & \forall \hat{\mu} \in \htrace{\hat{\genbd}}.
    \end{align}
    Using the Poincaré inequality for functions with vanishing trace, we similarly find a constant $C_{00}(\hat{\genbd}) > 0$ such that
    \begin{align}
    \label{eq:interpolation poincare proof - rescaled inequality II}
        \htraceseminorm{\hat{\mu}^\star}{\partial \hat{\domain}} \leq C_{00}(\hat{\genbd}) \sqrt{\Ltwonorm{\hat{\mu}}{\hat{\genbd}} \Ltwonorm{\nabla_t\hat{\mu}}{\hat{\genbd}}}, & &\forall \hat{\mu} \in \htrace{\hat{\genbd}}.
    \end{align} 
    This finishes the first step.

    For the second step, we note that \cref{lemma:scaling of sobolev norms} implies that
    \begin{align*}
        \Ltwonorm{\mu}{\genbd} &= |\genbd|^{-\frac{1}{2}} \Ltwonorm{\hat{\mu}}{\hat{\genbd}}, & \Ltwonorm{\mu - \avg{\mu}{\genbd}}{\genbd} &= |\genbd|^{-\frac{1}{2}} \Ltwonorm{\hat{\mu} - \avg{\hat{\mu}}{\hat{\genbd}}}{\hat{\genbd}},
        \\
        \htraceseminorm{\mu}{\genbd} &= |\genbd|^{\frac{2 - n}{2(n - 1)}} \htraceseminorm{\hat{\mu}}{\hat{\genbd}}, &
        \htraceseminorm{\mu - \avg{\mu}{\genbd}}{\genbd} &= |\genbd|^{\frac{2 - n}{2(n - 1)}} \htraceseminorm{\hat{\mu} - \avg{\hat{\mu}}{\hat{\genbd}}}{\hat{\genbd}},
        \\
        \Ltwonorm{\nabla_t \mu}{\genbd} &= |\genbd|^{\frac{3 - n}{2(n-1)}} \Ltwonorm{\nabla_t \hat{\mu}}{\hat{\genbd}}.
    \end{align*}
    Observe that
    \begin{align*}
        -\frac{1}{2} + \frac{3 - n}{2(n -1)} = \frac{2 - n}{n-1}.
    \end{align*}
    Therefore, multiplying \cref{eq:interpolation poincare proof - rescaled inequality I} and \cref{eq:interpolation poincare proof - rescaled inequality II} by $|\genbd|^{\frac{2 - n}{2(n-1)}}$ yields
    \begin{align*}
        \htraceseminorm{\mu - \avg{\mu}{\genbd}}{\genbd} \leq C(\hat{\genbd}) \sqrt{\Ltwonorm{\mu - \avg{\mu}{\genbd}}{\genbd} \Ltwonorm{\nabla_t \mu}{\genbd}}, & & \forall \mu \in \htrace{\genbd},
    \end{align*}
    and
    \begin{align*}
        \htraceseminorm{\mu^\star}{\partial \domain} \leq C_{00}(\hat{\genbd}) \sqrt{\Ltwonorm{\mu}{\genbd} \Ltwonorm{\nabla_t\mu}{\genbd}}, & &\forall \mu \in \htrace{\genbd},
    \end{align*}
    which finishes the proof.

\end{proof}

\subsection{Norm equivalence}
\label{subsec:norm equivalence}
We state in this section the general result on the norm equivalence between the natural and intrinsic norms in trace spaces with explicit dependence on the properties of the underlying Lipschitz boundary. A detailed proof can be found in \cite{leoni_first_2017}.

\begin{theorem}[Theorem 18.40 in \cite{leoni_first_2017}]
\label{thm:norm equivalence}
Let $\domain \subset \R^n, n \geq 2$ be an open set with uniformly Lipschitz boundary $\partial \domain$. Let $\epsilon, L > 0$ and $M \in \N$ be the constants in \cref{def:uniformly lipschitz cts boundary}. Then, for every $\mu \in \htrace{\partial \domain}$, there exist a constant $c = c(n)$ and $u \in \hone{\domain}$, such that
\begin{align*}
    \Ltwonorm{u}{\domain} \leq (M \epsilon)^{1/2} \Ltwonorm{\mu}{\partial \domain},
\end{align*}
and
\begin{align}
\label{eq:general norm equivalence}
    \honeseminorm{u}{\domain} \leq c M (1 + L)^{3 + n / 2} \epsilon^{-1/2} \Ltwonorm{\mu}{\partial \domain} + c M (1 + L)^{2 + (n+1)/2} \htraceseminorm{\mu}{\partial \domain}.
\end{align}
\end{theorem}

\subsection{Extensions}
\label{subsec:extensions}
We restate here the extension result from \cite{weder_extension_2025}, where we use the fact that the extension operator is constructed locally and, therefore, we may assume the extension to be supported in an intermediate compact subset:

\begin{theorem}
\label{thm:boundary extension operator}
    Let $\genbd \subset \R^n$ be a compact $(n-1)$-dimensional Lipschitz manifold and $\defbd \subset \genbd$ an $(n-1)$-dimensional submanifold with boundary. Furthermore, we assume that there exists a subset $\genbd_0 \subset \genbd$, such that
    \begin{align*}
        \defbd \subset \genbd_0 \subset \genbd, & \text{ and } \partial \defbd \cap \partial \genbd = \emptyset.
    \end{align*}
    Finally, we assume that the boundary $\partial \defbd$ of $\defbd$ is an $(n-2)$-dimensional Lipschitz manifold.

    Then, there exists a continuous extension operator $\extop{\defbd}{\genbd}: \htrace{\defbd} \to \htracedbz{\genbd}$, such that for $\mu \in \htrace{\defbd}$,
    \begin{align*}
        \Ltwonorm{\extop{\defbd}{\genbd}(\mu)}{\genbd}^2 \lesssim \Ltwonorm{\mu}{\defbd}^2,
    \end{align*}
    and
    \begin{align*}
        \htraceseminorm{\extop{\defbd}{\genbd}(\mu)}{\genbd}^2 \lesssim |\defbd|^{\frac{-1}{n-1}} \Ltwonorm{\mu}{\defbd}^2 + \htraceseminorm{\mu}{\defbd}^2
    \end{align*}
\end{theorem}

\begin{proof}
    For the construction of the extension operator, we refer to \cite{weder_extension_2025}. In particular, in the proof, one can reduce the covering radius $\epsilon > 0$ if needed, such that the support of $\extop{\genbd}{\partial\domain}(\mu) \subset \genbd_0$ and we can restrict the extension to $\genbd$.
    Hence, choosing $k = n-1, p = 2$, and $s = 1/2$ in \cite{weder_extension_2025}[Theorem 2], we find
    \begin{align*}
        \Ltwonorm{\extop{\defbd}{\genbd}(\mu)}{\genbd}^2 \lesssim \Ltwonorm{\mu}{\defbd}^2,
    \end{align*}
    and
    \begin{align*}
        \htraceseminorm{\extop{\defbd}{\genbd}(\mu)}{\genbd}^2 \lesssim |\defbd|^{\frac{-1}{n-1}} \Ltwonorm{\mu}{\defbd}^2 + \htraceseminorm{\mu}{\defbd}^2,
    \end{align*}
    where we have dropped higher-order terms in $|\defbd|$ for simplicity.
    
    Finally, using that $\support{\extop{\defbd}{\genbd}(\mu)} \subset \genbd_0 \subset \genbd$, we find for the $\htracedbz{\defbd}$-norm that
    \begin{align*}
        \int_{\partial \domain \setminus \genbd} \int_\genbd \frac{|\extop{\defbd}{\genbd}(\mu)(x)|^2}{|x - y|^n} \dd{s(x)} \dd{s(y)}  = \int_{\partial \domain \setminus \genbd} \int_{\genbd_0} \frac{|\extop{\defbd}{\genbd}(\mu)(x)|^2}{|x - y|^n} \dd{s(x)} \dd{s(y)}
        \\
        \lesssim \frac{|\partial \domain \setminus \genbd|}{\dist(\partial\domain\setminus\genbd, \genbd_0)^n} \Ltwonorm{\extop{\defbd}{\genbd}(\mu)}{\genbd}^2 \lesssim |\defbd|^{\frac{-1}{n-1}} \Ltwonorm{\extop{\defbd}{\genbd}(\mu)}{\genbd}^2,
    \end{align*}
    using that $|\partial \domain \setminus \genbd| \simeq |\defbd|$ and $\dist(\partial\domain\setminus\genbd, \genbd_0) \simeq |\defbd|^{\frac{1}{n-1}}$ under rescaling. This concludes the proof.
\end{proof}

\bibliographystyle{siamplain}
\bibliography{references}

\end{document}